\documentclass{siamart190516}

\usepackage{amsmath,amssymb,amsfonts,bm,comment,pgf}
\usepackage{algorithmic}
\Crefname{ALC@unique}{Line}{Lines} 

\usepackage{tabularx}
\usepackage{longtable}
\usepackage{multirow}
\usepackage{makecell}
\usepackage{booktabs}
\newcommand{\one}{{\bm 1}}
\newcommand{\z}{}
\newcommand{\cu}{\mathbf i}
\newcommand{\E}{\mathcal{E}}

\DeclareMathOperator*{\dotle}{\,\dot{\le}\,}

\begin{document}
\title{Solving quadratic matrix equations arising in random walks in the quarter plane\thanks{Research partially supported by INdAM-GNCS}}
\author{Dario A. Bini\thanks{Dipartimento di Matematica, Universit\`a di Pisa, Italy, (\email{dario.bini@unipi.it})} \and
Beatrice Meini\thanks{Dipartimento di Matematica, Universit\`a di Pisa, Italy, (\email{beatrice.meini@unipi.it})} \and 
Jie Meng\thanks{Department of Mathematics, Pusan National University, Busan, South Korea (\email{mengjie@pusan.ac.kr})}}
\maketitle

\begin{abstract}
Quadratic matrix equations of the kind $A_1X^2+A_0X+A_{-1}=X$ are encountered in the analysis of Quasi--Birth-Death stochastic processes where the solution of interest is the minimal nonnegative solution $G$. In many queueing models, described by random walks in the quarter plane, 
the coefficients $A_1,A_0,A_{-1}$ are infinite tridiagonal matrices with an almost Toeplitz structure. Here, we analyze some fixed point iterations, including Newton's iteration, for the computation of $G$ and introduce effective algorithms and acceleration strategies which fully exploit the Toeplitz structure of the matrix coefficients and of the current approximation. Moreover, we provide
a structured perturbation analysis for the solution $G$. The results of some numerical experiments which demonstrate the effectiveness of our approach are reported.

\end{abstract}



\section{Introduction} Random walks in the quarter plane describe a wide variety of 
two-queue models with various service policies such as nonpreemptive
priority, $K$-limited service, server vacation and server setup
\cite{ozawa19}. Models of this kind concern, for instance, bi-lingual
call centers \cite{stanford}, generalized two-node Jackson networks
\cite{ozawa18}, two-demand models \cite{flatto}, two-stage inventory
queues \cite{haque}, and more.

A theoretical analysis of stability, of tail decay rates and of other
asymptotic properties has been carried out by several authors, in
particular in \cite{koba-miya}, \cite{miya1},  \cite{miya-zhao},
\cite{ozawa19}, and in the book \cite{fayolle:book}, in which the
invariant measure and the transient behavior are investigated by
means of analytic and functional tools.

 A different approach is based on representing a random walk in the
 quarter plane as a 2-dimensional Quasi--Birth-Death (QBD) stochastic
 process. This latter framework, based on the matrix analytic approach
 of \cite{neuts}, allows to express the invariant probability measure,
 and other quantities of interest for the stochastic model, in terms of
 a solution of suitable quadratic matrix equations. This provides a
 further tool for the theoretical analysis 
 \cite{lat:varese}, \cite{lr:book} and paves the way for the design of effective
 algorithms based on the numerical solution of quadratic matrix
 equations.

 In fact, relying on the matrix analytic theory of \cite{neuts}, the
 problem of computing the invariant probability measure of a QBD
 process is reduced to computing the minimal nonnegative solution $G$
 and $R$ of the two matrix equations
 \begin{eqnarray}
A_1X^2+A_0X+A_{-1}=X\label{eq:G},\\
A_1+XA_0+X^2A_{-1}=X\label{eq:R},
 \end{eqnarray}
respectively, where the coefficients $A_{-1},A_0,A_1$ are nonnegative matrices such that $A_{-1}+A_0+A_1$ is row-stochastic and $X$ is the unknown.
We say that a matrix $X$ is nonnegative, and we write $X\ge 0$,  if its entries are nonnegative. Moreover we say that  a solution $X$ of a matrix equation is minimal nonnegative if $X\ge 0$ and for any other nonnegative solution $Y$ it holds $Y-X\ge 0$. For more details in this regard, we refer the reader to the books \cite{blm:book}, \cite{lr:book}, and \cite{neuts}. 

In the case where the coefficients are finite dimensional, several algorithms have been introduced to compute $G$ and $R$. They include fixed point iterations and doubling algorithms like Logarithmic Reduction and Cyclic Reduction (CR) \cite{bm:cr}, \cite{blm:book},  \cite{lr:book}. 

In the case of 2-dimensional QBDs the coefficients $A_{-1},A_0$, $A_1$ are semi-infinite and have a special structure, more precisely,
\begin{equation}\label{eq:blocks}
A_{i}=\begin{bmatrix}
b_{i,0}&b_{i,1}\\
a_{i,-1}&a_{i,0}&a_{i,1}\\
&a_{i,-1}&a_{i,0}&a_{i,1}\\
&&\ddots&\ddots&\ddots\\
\end{bmatrix},\quad i=-1,0,1,
\end{equation}
where $a_{i,j}\ge0$, $b_{i,j}\ge0$ and $\sum_{i,j=-1}^1 a_{i,j}=1$, $\sum_{i=-1}^1\sum_{j=0}^1 b_{i,j}=1$.
These blocks belong to the class of matrices representable in the form
$A=T(s)+E$ where $T(s)$ is the Toeplitz matrix associated with the symbol $s(z)=\sum_{i\in\mathbb Z}s_iz^i$, that is, $(T(s))_{i,j}=s_{j-i}$, and $E=(e_{i,j})$ is  such that $v_i=\sum_{j}|e_{i,j}|$ is finite and $\lim_i  v_i=0$. 
The matrix $T(s)$ is called {\em Toeplitz part} while $E$ is called {\em correction}.
Here $s(z)$ is a function belonging to the Wiener class $\mathcal W=\{f(z)=\sum_{i\in\mathbb Z}f_iz^i, ~ \|f\|_w:=\sum_{i\in\mathbb Z}|f_i|<\infty\}$. In particular, for the matrix in \eqref{eq:blocks} it is easy to check that $A_i=T(a_i)+E_i$ where $a_i(z)=\sum_{j=-1}^1a_{i,j}z^j$ and $E_i$ has zero entries except for the first row which is equal to $[b_{i,0}-a_{i,0},b_{i,1}-a_{i,1},0,\ldots]$. 
Matrices of this kind are called {\em Quasi-Toeplitz (QT)} in \cite{bmm}.

The case of QBD with infinite blocks has been initially investigated in   \cite{lat:varese}, \cite{latouche11}, and \cite{latouche02}, by reducing the problem to finite size relying on truncation and augmentation of the blocks. However, this approach does not lead to reliable computational techniques since the result of the numerical computation is strongly dependent on the way the infinite matrices have been truncated. 
In fact, the models obtained by truncating the infinite dimensional problem may have asymptotic properties, like the decay rate, which are not consistent with the original problem \cite{kroese}, \cite{sakuma}.

More recently, conditions under which the solution $G$ of \eqref{eq:G}
can be represented as the sum $G=T(g)+E_g$ are given in \cite{bmmr19},
so that, despite the solution $G$ has infinitely many entries, it can
be represented up to any arbitrary approximation error by using a
finite number of parameters.  Moreover, in \cite{bmm} and \cite{bmmr},
by using the structure properties of QT matrices, the algorithm of
Cyclic Reduction has been extended to the case of infinite matrices.
This algorithm still keeps a fast convergence speed in terms of number
of iterative steps. However, in certain cases the cost of each step
becomes extremely large due to the cost of certain operations with QT
matrices, like matrix inversion and the compression of the correction
part.  Another drawback of CR is that this iteration is not
self-correcting.

In this paper, we propose and analyze some fixed point iterations
which have a low cost per step and, unlike CR, are self-correcting and
allow to keep separated the computation of the Toeplitz part $T(g)$
and the correction part $E_g$.  In fact, we show that the symbol
$g(z)$ defining the Toeplitz part satisfies the functional equation
\[
a_1(z)g(z)^2+a_0(z)g(z)+a_{-1}(z)=g(z),~~~|z|=1.
\]
We use this property to design an algorithm based on evaluation and
interpolation at the roots of 1 for approximating the coefficients of
$g(z)$, which is extremely fast and allows for an automatic control of
the number of interpolation points, according to the desired
approximation error.

The correction part $E_g$ is obtained by simply applying fixed point
iterations. We consider three iterations of the kind $X_{k+1}=F(X_k)$,
$k=0,1,\ldots$, defined by suitable functions $F_1$, $F_2$, $F_3$,
where $F_1$ requires no matrix inversion, $F_2$ requires to compute an
inverse matrix once for all, while $F_3$ requires one inversion per
step.  These iterations are well known in the case of QBD with
finitely many phases \cite{blm:book}, \cite{lr:book}, \cite{meini},
and are here extended to coefficients with the QT structure
\eqref{eq:blocks}.

We show that, under mild assumptions, starting with $X_0=0$, the
sequences generated by $F_1,F_2,F_3$ converge monotonically and
linearly to $G$ in the infinity norm. Moreover, we prove that the rate
of convergence of the sequence generated by $F_3$ is better than the
rate of the sequence generated by $F_2$, which in turn is better than
that generated by $F_1$. We prove that if $X_0$ is row-stochastic then
all the matrices $X_k$ are row-stochastic and the rate of convergence
of each of the three iterations is better than that obtained with
$X_0=0$. Numerical experiments show also the evidence that for
$X_0=T(g)$ the rate of convergence of the three sequences is even
better.

Then we adapt Newton's iteration, in the form given by
\cite{latoucheN}, to the case of QT coefficients. In order to solve
the Sylvester equation arising at each step of Newton's iteration we
rely on the solver introduced in \cite{robol}. Under mild conditions,
we prove that, for $X_0=0$, convergence holds in the infinity norm, is
monotonic and quadratic.

In order to evaluate an {\em a posteriori} bound on the approximation error
in the computation of $G$, we also perform the analysis of the
structured condition number. More specifically, we provide
perturbation results related to perturbations of the Toeplitz part and
of the correction part in the matrix coefficients $A_i$, $i=-1,0,1$,
and we estimate the consequent variation of the solution $G$ and of
its Toeplitz part $T(g)$.

Numerical experiments are reported which show the effectiveness of our
approach and the reliability of our algorithms with respect to the
algorithm CR.  In particular we show that in certain cases the combination
of  Newton's iteration and cyclic reduction provides a substantial
acceleration of the convergence.


The paper is organized as follows: in Section \ref{sec:prel} we recall some preliminary properties and concepts useful for the analysis of the problem;  in Section \ref{sec:g} we present an algorithm for computing the Toeplitz part $T(g)$ of the solution; Section \ref{sec:fxp} deals with the analysis of three fixed-point iterations applied to infinite QT matrices,  while Section
\ref{sec:newton} concerns the algorithmic analysis of Newton iteration; 
in Section \ref{sec:perturb} we carry out the analysis of the conditioning by providing some perturbation results, while in Section \ref{sec:exp} we present and discuss some numerical experiments which show the effectiveness of our approach.

\section{Preliminaries}\label{sec:prel}
Let $\ell^\infty$ be the set of sequences $ x=(x_i)_{i\in\mathbb
  N}$ such that $\| x\|_\infty:=\sup_i |x_i|$ is finite. Consider
the set of matrices $A=(a_{i,j})$ such that the application $
x\to y=A x$, where $y_j=\sum_{j=1}^\infty a_{i,j} x_j$, defines
a linear operator from $\ell^\infty$ to $\ell^\infty$. Denote this set
by $\mathcal L_\infty$ and define the induced norm
$\|A\|_\infty=\sup_{\| x\|_\infty=1}\|A x\|_\infty$. It can be
verified that $\|A\|_\infty=\sup_i\sum_{j=1}^\infty |a_{i,j}|$.
Recall that $\mathcal L_\infty$ is a Banach algebra, that is, it is
closed under the row-by-column product, the norm satisfies
$\|AB\|_\infty\le\|A\|_\infty\cdot\|B\|_\infty$ for any
$A,B\in\mathcal L_\infty$, and the normed space is complete.

We introduce the following notation
\begin{equation}\label{eq:ab}
\begin{aligned}
&a_i(z)=a_{i,-1}z^{-1}+a_{i,0}+a_{i,1}z\\
&b_i(z)=b_{i,0}+b_{i,1}z,
\end{aligned}
\end{equation}
so that we may write $A_i=T(a_i)+E_i$, for $i=-1,0,1$, where $E_i$ has
null entries except for those in the first row which are equal to
$[b_{i,0}-a_{i,0},b_{i,1}-a_{i,1},0,\ldots]$.  We assume that the
entries of $A_i$ are nonnegative and $(A_{-1}+A_0+A_1)\one=\one$,
where $\one$ is the vector of all ones of appropriate dimension.  It
is known \cite{lr:book}, \cite{taka} that under these conditions,
there exist the minimal nonnegative solutions $R$ and $G$ of
\eqref{eq:G} and \eqref{eq:R}, respectively, and the Laurent matrix
polynomial $\varphi(z)=z^{-1}A_{-1}+A_0-I+zA_1$ admits the
factorization
\[
\varphi(z)=-(I-zR)W(I-z^{-1}G)
\]
where
\begin{equation}\label{eq:wh}
\begin{aligned}
&A_1=RW,\quad A_{-1}=WG,\quad A_0=I-W-RWG,\\
&W=I-A_0-A_1G=I-A_0-RA_{-1},\\
&G\one\le\one.
\end{aligned}
\end{equation}

Observe that if $a_{-1}(1)=0$, i.e., $A_{-1}=e_1w^T$,
$w^T=(b_{-1,0},b_{-1,1},0,\ldots)\ne 0$, then the minimal nonnegative
solution $G$ of equation \eqref{eq:G} can be expressed in the form
$G=\one v^T$ where $v=\frac1{b_{-1}(1)}w$. Therefore, without loss
of generality we may assume that $a_{-1}(1)>0$.

The following result is valid if $a_{-1}(1)>0$ and $b_{-1}(1)>0$, that is, $A_{-1}\one>0$.

\begin{lemma}\label{lem:1}
Assume $A_{-1}\one>0$ and define
\begin{equation}\label{eq:gamma}
\theta = \min\{a_{-1}(1),b_{-1}(1)\},\quad
\gamma=\max\left\{\frac{a_1(1)}{a_{-1}(1)},\frac{b_1(1)}{b_{-1}(1)}\right\}.
\end{equation}
Then the matrix $W=I-A_0-A_1G$ is invertible in $\mathcal L_\infty$, has nonnegative inverse, and $\|W^{-1}\|_\infty\le \frac 1{1-\|(A_0+A_1)\one\|_\infty}=\frac 1\theta$. Moreover,
 $\|W^{-1}RW\|_\infty\le\gamma$. If $A_{-1}\one>A_1\one$ then $\gamma<1$.

\end{lemma}
\begin{proof}
Observe that $\|A_0+A_1G\|_\infty=\|(A_0+A_1G)\one\|_\infty\le\|(A_0+A_1)\one\|_\infty=\|(I-A_{-1})\one\|_\infty=1-\theta<1$,  since 
$(A_0+A_1)\one=(I-A_{-1})\one$, $A_{-1}\one>0$ and $A_{-1}\one=(b_{-1}(1),a_{-1}(1),a_{-1}(1),\ldots)^T$.  Therefore $W^{-1}\in\mathcal L_\infty$ and is nonnegative, being $W^{-1}=\sum_{i=0}^\infty (A_0+A_1G)^i$ and
\[
\|W^{-1}\|_\infty = \|\sum_{k=0}^\infty (A_0+A_1G)^k\one\|_\infty\le\frac1{1-\|A_0+A_1\|_\infty}=\frac 1\theta.
\] 
Now we show that $\|W^{-1}RW\|_\infty\le\gamma$. Since 
$W^{-1}RW=W^{-1}A_1$ by \eqref{eq:wh}, it is sufficient to consider $\|W^{-1}A_1\|_\infty=\|W^{-1}A_1\one\|_\infty$. By definition of $\gamma$ we have $A_1\one\le\gamma A_{-1}\one$, therefore
\[
W^{-1}A_1\one\le \gamma W^{-1}A_{-1}\one=\gamma G\one \le \gamma\one,
\] 
where we used the fact that $W^{-1}A_{-1}=G$.
Thus we have $\|W^{-1}RW\|_\infty\le\gamma$. Since $A_1\one=(b_1(1),a_1(1),a_1(1),\ldots)^T$ and $A_{-1}\one=(b_{-1}(1),a_{-1}(1),a_{-1}(1),\ldots)^T$ then $A_{-1}\one>A_1\one$ implies $\gamma<1$.
\end{proof}

Define $\mathcal W=\{f(z)=\sum_{i\in\mathbb Z}f_iz^i:\quad \|f\|_w:=\sum_{i\in\mathbb Z}|f_i|<\infty\}$. Consider the following class
\[
\mathcal{QT}:=\{A=T(f)+E\}
\]
where $f(z)\in\mathcal W$, the matrix $E=(e_{i,j})\in\mathcal L_\infty$ is such that $\lim_i v_i=0$, where $v_i=\sum_{j=1}^\infty |e_{i,j}|$.

Observe that $A_i\in\mathcal{QT}$ for $i=-1,0,1$,  moreover,
in \cite{bmmr19} it is shown that $\mathcal{QT}$ is an algebra with the infinity norm, and the matrices $W,G$ and $R$ in \eqref{eq:wh} belong to $\mathcal{QT}$ if $A_{-1}\one>A_1\one$ or if $A_{-1}\one\ge A_1\one>0$. 
More precisely we have the following result

\begin{theorem}\label{thm:A}

The minimal nonnegative solution $G$ of the matrix equation
\eqref{eq:G} can be written as $G=T(g)+E_g$ where
$g(z)=\sum_{i\in\mathbb Z}g_iz^i \in\mathcal W$ is such that $g_i\ge
0$ and $\|g\|_w=g(1)\le 1$. Moreover, for any $z$ such that $|z|=1$,  $g(z)$ is a  solution of minimum modulus of the quadratic
equation
\begin{align}\label{ag}
a_{-1}(z)+a_0(z)\lambda+a_1(z)\lambda^2=\lambda.
\end{align}
This solution is unique if there exists $j$ such that $a_{i,j}\ne 0$ for at least two different values of $i$. 
If 
\[
A_{-1}\one>A_1\one,\quad \hbox{ or } \quad A_{-1}\one\ge A_1\one>0,
\] 
then $G\in\mathcal{QT}$, $G\one=\one$  and $g(1)=1$.
Conversely, if $G\one=\one$ and  $G\in\mathcal{QT}$  then
$a_{-1}(1)\ge a_1(1)$ and $g(1)=1$.
\end{theorem}

Observe that the condition 
\begin{equation}\label{eq:assumption1}
A_{-1}\one>A_1\one
\end{equation}
is equivalent to $a_{-1}(1)>a_1(1)$ and $b_{-1}(1)>b_1(1)$. In the
following we assume that \eqref{eq:assumption1} holds and that there exists $i$ such that $a_{i,j}\ne 0$ for at least two values of $j$. The latter condition is very mild.

\section{Computing  the symbol $g(z)$}\label{sec:g}
Relying on Theorem \ref{thm:A}, we provide an algorithm, based on the
evaluation/in\-ter\-po\-la\-tion at the roots of 1, for computing an
approximation $\hat g_i$, $i=-n+1,\ldots,n$, to the coefficients $g_i$
of $g(z)$, where $n$ is such that $|\hat g_i-g_i|\le \epsilon/(2n)$, for any $i$ and
for a given tolerance $\epsilon>0$. In this analysis we may relax the
assumption $a_{-1}(1)>a_1(1)$ so that the result holds in general.

Let $n>0$ be an integer and set $m=2n$. Define
$\omega_m=\cos\frac{2\pi}{m}+\cu \sin \frac{2\pi}{m}$ a principal
$m$th root of 1, where $\cu$ is the imaginary unit such that
$\cu^2=-1$. Rewrite $g(z)$ as
\begin{equation}\label{eq:repr}
g(z)=\sum_{j=-n+1}^n g_jz^j+\sum_{j=-n+1}^n\sum_{k\ge 1}(z^{mk+j}g_{mk+j}+z^{-mk-j+1}g_{-mk-j+1}).
\end{equation}
Since $\omega_m^{km}=1$, from \eqref{eq:repr} we have
\[
g(\omega_m^i)=\sum_{j=-n+1}^n g_j\omega_m^{ij}+\sum_{j=-n+1}^n(\omega_m^{ij}\sum_{k\ge 1}g_{mk+j}+\omega_m^{-i(j-1)}\sum_{k\ge 1}g_{-mk-j+1}).
\]
Therefore, the Laurent polynomial defined by
\begin{equation}\label{eq:hatp}\begin{aligned}
&\widehat g(z)=\sum_{j=-n+1}^n g_jz^j+\sum_{j=-n+1}^n(z^{j}\hat g_j^+ + z^{-j+1}\hat g_{-j+1}^-),\\
&\hat g_j^+=\sum_{k\ge 1}g_{mk+j},\quad \hat g_{-j+1}^-=
\sum_{k\ge 1}g_{-mk-j+1},~~j=-n+1,\ldots,n,
\end{aligned}
\end{equation}
is such that $g(\omega_m^i)=\hat g(\omega_m^i)$, that is, it interpolates $g(z)$ at the $m$-th roots of 1.

The following lemma provides a bound to the tail of the Laurent series $g(z)$ and extends to the case of Laurent series a similar property proved in \cite{pwcr} valid for power series.

\begin{lemma}\label{lem:3.1}
Let $g(z)$ be the solution of minimum modulus of equation \eqref{ag}. Let $\hat g(z)=\sum_{j=-n+1}^n \hat g_jz^j$ be the Laurent polynomial interpolating $g(z)$ at the $m$-th roots of 1, i.e., such that $g(\omega_m^i)=\hat g(\omega_m^i)$, $i=-n+1,\ldots,n$, where $m=2n$. If $g''(x)\in\mathcal W$, then $g''(1)\ge 0$ and
\begin{equation}\label{ineq}
g''(1)-\hat g''(1)\ge 2n\left(\sum_{j<-n+1}g_j+\sum_{j>n}g_j\right),
\end{equation}
moreover $0\le \hat g_j-g_j\le \frac1{2n}(g''(1)-\hat g''(1))$, for $j=-n+1,\ldots,n$. 
\end{lemma}
\begin{proof} Since the coefficients $g_i$ are nonnegative then also $g''(z)$ has nonnegative coefficients, moreover, since $g''(z)\in\mathcal W$, then the series $g''(1)$ is absolutely convergent and $g''(1)\ge 0$.
Thus, from the representation \eqref{eq:repr}, in view of \eqref{eq:hatp}, we deduce that
\[
g''(1)-\hat g''(1)=\sum_{j=-n+1}^n \sum_{k\ge 1}\left( 
g_{mk+j}\alpha_{j,k}+g_{-mk-j+1}\alpha_{j,k}\right),
\]
where $\alpha_{j,k}=(mk+j)(mk+j-1)-j(j-1)$.
The inequality \eqref{ineq}  follows from the nonnegativity of the coefficients and from the property $\alpha_{j,k}\ge m$, valid for $k\ge 1,~j=-n+1,\ldots,n$ which can be verified by a direct inspection. The bound on $\hat g_j-g_j$ follows from \eqref{ineq} since $\hat g_j = g_j+\hat g_j^+ + \hat g_j^-$ and $\hat g^+_j+\hat g_j^- \le \sum_{j<-n+1}g_j+\sum_{j>n}g_j$ in view of \eqref{eq:hatp}.
\end{proof}

Observe that $\hat g''(1)$ is computable once the coefficients of the polynomial $\hat g(z)$ have been computed. Moreover, the value of $g''(1)$ is computable even though $g(z)$ is not known. In fact, by taking the second derivative in the equation obtained by replacing $\lambda$ with $g(z)$ in \eqref{ag}, i.e.,
\[
a_1(z)g(z)^2+(a_0(z)-1)g(z)+a_{-1}(z)=0, 
\]
for $z=1$, the value of $g''(1)$ can be easily expressed in terms of $a_i(1)$, $a_i'(1)$ and $a_i''(1)$. More precisely, by taking the first derivative we obtain
\[
a_1'(z)g(z)^2+2g(z)g'(z)a_1(z)+(a_0(z)-1)g'(z)+a_0'(z)g(z)+a_{-1}'(z)=0,
\]
which yields
\begin{equation}\label{eq:g1}
g'(1)=\frac{a_1'(1)g(1)^2+a_0'(1)g(1)+a_{-1}'(1)}{1-2a_1(1)g(1)-a_0(1)},\quad g(1)=\min(1,a_{-1}(1)/a_1(1)).
\end{equation}
By taking the second derivative for $z=1$, we get
\[\begin{aligned}
a_1''(1)g(1)^2+& 2a_1'(1)g'(1)g(1)+
2a_1'(1)g'(1)g(1)+2a_1(1)g'(1)^2\\
+&2a_1(1)g''(1)g(1)+
(a_0(1)-1)g''(1)+a_0'(1)g'(1)\\
+& a_0''(1)g(1)+a_0'(1)g'(1)
+a_{-1}''(1)=0
\end{aligned}\]
which yields
\begin{equation}\label{eq:g2}
\begin{aligned}
g''(1)=&\left[ a_{-1}''(1)+a_0''(1)g(1)+a_{1}''(1)g(1)^2+2a_1(1)g'(1)^2\right. \\
+&\left. 2g'(1)(2g(1)a_1'(1)+a_0'(1))\right] /(1-2a_1(1)g(1)-a_0(1)).
\end{aligned}
\end{equation}

Lemma \ref{lem:3.1} provides an {\em a posteriori} bound to the error in the approximation of the Laurent series $g(z)$ together with a stop condition for the following evaluation interpolation algorithm for computing the coefficients of $g(z)$.

\begin{algorithm}[H]
\caption{Approximation of $g(z)$}
\label{alg:g}
 \begin{algorithmic}[1] 
 \REQUIRE{The coefficients of $a_i(z)$, $i=-1,0,1$  and a tolerance $\epsilon>0$.}
\ENSURE{Approximations $\hat g_i$, $i=-n+1,\ldots,n$,   to the coefficients $g_i$ of $g(z)$
 such that $\hat g_i-g_i\le \epsilon/(2n)$.}

\STATE{Set $n=4$, and compute $g(1)=\min(1,a_{-1}(1)/a_1(1))$,  $g'(1)$ and $g''(1)$ by means of \eqref{eq:g1} and \eqref{eq:g2}; 
}
\STATE{Set $m=2n$, $\omega_m=\cos\frac{2\pi}m+\cu\sin\frac{2\pi}m$, and
evaluate $a_{-1} (z),~ a_0 (z),~ a_1 (z)$ at $z = \omega^i_m$, $i=-n+1,\ldots,n$;
}
\STATE{For $i=-n+1,\ldots,n$, compute the solution $\lambda_i$ of minimum modulus of the quadratic equation \eqref{ag}, where $z =\omega_m^i$;
}
\STATE{Interpolate the values $\lambda_i$, $i=-n+1,\ldots,n$ by means of FFT and obtain  the coefficients $\hat g_i$ of the Laurent polynomial $\hat g(z)=\sum_{i=-n+1}^n \hat g_iz^i$ such that $g(\omega_m^i)=\hat g(\omega_m^i)$, $i=-n+1,\ldots,n$;
}
\STATE{Compute $\delta_m=g''(1)-\hat g''(1)$, where  $\hat g''(1)=\sum_{i=-n+1}^n i(i-1)\hat g_i$;
}
\STATE{If $\delta_m/m\le\epsilon$ then exit, else set $n=2n$ and continue from Step 2.
}
\end{algorithmic}
\end{algorithm}

Observe that the error bound converges to zero at least as $O(1/n)$. If the function $g(z)$ is analytic in a neighborhood of the unit circle, then its coefficients decay exponentially to zero
 \cite{henrici} so that also the the bound on the error converges exponentially to zero.
It is also interesting to observe that, for $n\to\infty$, the convergence of the coefficients of $\hat g(z)$ to the corresponding coefficients of  $g(z)$ is monotonic.

Finally observe that the overall computational cost of this algorithm is $O(n\log n)$ arithmetic operations. In order to complete the computation of $G$ it remains to approximate the correction $E_g$. In view of the fact that $E_g$ has entries $e^{(g)}_{i,j}$ such that $v_i=\sum_{j}|e^{(g)}_{i,j}|$ is finite and $\lim_i v_i=0$, we can approximate $E_g$ with a finite number of parameter within an error bound $\epsilon$. This computation is performed by means of functional iteration and is analyzed in the next section.

\section{Fixed point iterations}\label{sec:fxp}
In this section we analyze the convergence of sequences generated by a functional iteration of the kind $X_{k+1}=F(X_k)$, $k=0,1,\ldots$, where $F(X)$ is a matrix function such that $G=F(G)$ where $G$ is the minimal nonnegative solution of~\eqref{eq:G}. More precisely, we will consider the following cases
\begin{equation}\label{eq:fixp}
\begin{aligned}
&F_1(X)=A_{-1}+A_0X+A_1X^2,\\
&F_2(X)=(I-A_0)^{-1}(A_{-1}+A_1X^2),\\
&F_3(X)=(I-A_0-A_1X)^{-1}A_{-1},
\end{aligned}
\end{equation}
while Newton iteration is considered in the next section. 

 In the case where $A_{-1},A_0,A_1$ are finite matrices, the convergence analysis of the sequences generated by the functions  \eqref{eq:fixp} has been performed in \cite{meini}. 
 Here, we extend the results of \cite{meini} 
 to the case of matrices of infinite size belonging to  $\mathcal L_\infty$. We need the following

\begin{lemma}\label{lem:2}
 Let $A_{-1}\one>A_1\one$, and 
 \begin{equation}\label{eq:sigma}
 \sigma=1-\min(a_{-1}(1)-a_1(1),b_{-1}(1)-b_1(1))<1.
 \end{equation}
 Let $H_1=
A_0+A_1+A_1G$. Then  $\|H_1\|_\infty\le\sigma$, so that $\|A_0\|_\infty\le\sigma$, $\|A_0+A_1G\|_\infty\le\sigma$. Therefore $I-A_0$ and $I-A_0-A_1G$ are invertible and, for $H_2=(I-A_0)^{-1}(A_1+A_1G)$, $H_3=(I-A_0-A_1G)^{-1}A_1$ we have
\[
\|H_3\|_\infty\le\|H_2\|_\infty\le\|H_1\|_\infty\le\sigma<1.
\]
\end{lemma}
\begin{proof}
We have $\|A_0+A_1+A_1G\|_\infty=\|(A_0+A_1+A_1G)\one\|_\infty$. Moreover, since by Theorem \ref{thm:A} we have $G\one=\one$ and $\one=(A_{-1}+A_0+A_1)\one$, then $(A_0+A_1+A_1G)\one = 
(A_0+A_1+A_1)\one=\one-(A_{-1}-A_1)\one\le\sigma \one$,  by definition of $\sigma$. This implies $\|A_0+A_1+A_1G\|_\infty\le\sigma$. Since $A_0,A_1,G$ are nonnegative then $\|A_0\|_\infty\le\|A_0+A_1+A_1G\|_\infty\le\sigma$ and $\|A_0+A_1G\|_\infty\le\sigma$. The matrices $I-A_0$ and $I-A_0-A_1G$ are invertible in $\mathcal L_\infty$ since, in general, if $I-B\in\mathcal L_\infty$ is such that $\|B\|_\infty<1$ then the series $\sum_{i=0}^\infty B^i$ has norm bounded by $1/(1-\|B\|_\infty)$  and coincides with $(I-B)^{-1}$. Concerning $H_2$  we have $\|H_2\|_\infty=\|H_2\one\|_\infty$. Moreover, 
\[
\begin{aligned}
H_2\one&=(I-A_0)^{-1}(A_1+A_1G)\one=(I-A_0)^{-1}(I-A_0-(A_{-1}-A_1G))\one\\
&=\one -(I-A_0)^{-1}(A_{-1}-A_1G)\one\le\one-(A_{-1}-A_1G)\one\\
&=(A_0+A_1+A_1G)\one=H_1\one,
\end{aligned}
\]
where we used the properties $(A_{-1}-A_1G)\one> 0$ and $(I-A_0)^{-1}\ge I$.
 Concerning $H_3$, since $A_1\one= (I-A_0-A_{-1})\one=(I-A_0-A_1G-(A_{-1}-A_1G))\one$ we have
\[
\begin{aligned}
H_3\one&=(I-A_0-A_1G)^{-1}A_1\one 
=\one -(I-A_0-A_1G)^{-1}(A_{-1}-A_1G)\one\\
&\le\one-(I-A_0)^{-1}(A_{-1}-A_1G)\one=H_2\one,
\end{aligned}
\]
where we used the fact that $G\one=\one$, $(A_{-1}-A_1G)\one>0$ and $(I-A_0-A_1G)^{-1}\ge (I-A_0)^{-1}$.
\end{proof}

Observe that, from Lemma \ref{lem:1} and from 
\eqref{eq:wh}, it follows that $H_3=W^{-1}RW$ and
$\|H_3\|_\infty=\|W^{-1}RW\|_\infty\le\gamma$ where $\gamma<1$ is
defined in \eqref{eq:gamma}. This provides a different bound on the
norm of $H_3$. Therefore we have
\begin{equation}\label{eq:sigmatau}
\|H_3\|_\infty\le \tau,\quad \tau=\min\{\gamma,\sigma\},
\end{equation}
with $\gamma$ and $\sigma$  defined in \eqref{eq:gamma} and \eqref{eq:sigma}, respectively.

We are ready to prove the following result which shows that the three sequences generated by \eqref{eq:fixp} starting with $X_0=0$ monotonically converge to $G$, convergence holds in the infinity norm and is linear. Moreover, the convergence of the third iteration is faster than that of the second one, while the convergence of the second iteration is faster than that of the first one.

\begin{theorem}\label{thm:a1} Assume that $A_{-1}\one>A_1\one$.
For $i\in\{1,2,3\}$ define $X_{k+1}^{(i)}=F_i(X_k^{(i)})$, $k=0,1,2,\ldots$, where $X_0^{(i)}=0$ and $F_i(X)$ are given in  \eqref{eq:fixp}. Then  
\begin{enumerate}
\item the three sequences $\{X_k^{(i)}\}$ are well defined,
\item   $0\le X_k^{(i)}\le X_{k+1}^{(i)}\le G$, 
\item  for the error $\E_k^{(i)}=G-X_k^{(i)}$ we have $\|\E_{k+1}^{(i)}\|_\infty\le\|H_i\|_\infty \|\E_k^{(i)}\|_\infty$, where $H_i$, $i=1,2,3$ are the matrices defined in Lemma \ref{lem:2}, so that $\lim_{k\to\infty}\|\E_k^{(i)}\|_\infty=0$.
\end{enumerate}
\end{theorem}
\begin{proof}
The first iteration is clearly well defined. The second is well defined since, according to Lemma \ref{lem:2}, the matrix $I-A_0$ is invertible in $\mathcal L_\infty$. The third iteration is well defined as long as the matrix $I-A_0-A_1X_k$ is invertible. On the other hand if $0\le X_k\le G$, the latter matrix is invertible in view of Lemma \ref{lem:2} since $\|A_0+A_1X_k\|_\infty\le\|A_0+A_1G\|_\infty$. 
In order to prove that $0\le X_k^{(i)}\le X_{k+1}^{(i)}\le G$ 
we use an induction argument. We prove it for the first iteration, i.e., for $i=1$, the same technique can be used for the other iterations. For notational simplicity we omit the superscript  and write $X_k$ in place of $X_k^{(1)}$. Since for $X_0=0$ we have $X_1=A_{-1}$, so that $0\le X_0\le X_1$ and  $G-X_1=G-A_{-1}=(A_0+A_1G)G\ge 0$.
For the inductive step, assume that $0\le X_{k-1}\le X_{k}\le G$. We first show that $0\le X_{k}\le X_{k+1}$. 
From $X_{k+1}=A_1X_k^2+A_0X_k+A_{-1}$ and from the property $0\le X_{k-1}\le X_{k}$ we get
\[
X_{k+1}\ge A_1X_{k-1}^2+A_0X_{k-1}+A_{-1}=X_k.
\]
Now consider
\begin{equation}\label{eq:ek}
\begin{split}
G-X_{k+1}&=A_1(G^2-X_k^2)+A_0(G-X_k)=\\
& =A_1((G-X_k)G+X_k(G-X_k))+A_0(G-X_k).
\end{split}
\end{equation}
Since $G- X_k\ge 0$ then also $G-X_{k+1}\ge 0$. 

Concerning the norm bounds to $\mathcal E_k$, for $\mathcal E_k=\mathcal E_k^{(1)}$ 
for the sequence defined by $F_1$, from \eqref{eq:ek}  we obtain
\begin{equation}\label{eq:F1}
\E_{k+1}=A_1 \E_k G+ A_1 X_k \E_k+A_0\E_k.
\end{equation}
Since $\E_k\ge 0$ for any $k$, then $\| \E_k\|_\infty = \| \E_k\one\|_\infty$, so that
\[
\begin{aligned}
\| \E_{k+1}\|_\infty &= \| \E_{k+1} \one\|_\infty\le \| A_1 \E_k \one+ A_1 X_k \E_k\one +A_0\E_k\one\|_\infty  \\
&\le \|(A_0+A_1+A_1 G)\E_k \one\|_\infty \le \| H_1\|_\infty \| \E_k\|_\infty.
\end{aligned}
\]
Similarly, concerning $F_2$  we obtain
\begin{equation}\label{eq:F2}
\E_{k+1}=(I-A_0)^{-1}A_1( \E_k G+  X_k \E_k),
\end{equation}
whence
\[
\begin{aligned}
\| \E_{k+1}\|_\infty &= \| \E_{k+1} \one\|_\infty\le \|(I-A_0)^{-1}A_1(I+X_k)\E_k\one\|_\infty \\
&\le \|(I-A_0)^{-1}(A_1+A_1G)\E_k\one\|_\infty
 \le \| H_2\|_\infty \| \E_k\|_\infty.
\end{aligned}
\]
Concerning $F_3$, we have
\begin{equation}\label{eq:F3}
\E_{k+1}=(I-A_0-A_1X_k)^{-1}A_1 \E_k G,
\end{equation}
whence
\[
\|\E_{k+1}\one\|_\infty\le\|(I-A_0-A_1G)^{-1}A_1\E_k\one\|_\infty\le\|H_3\|_\infty \|\E_k\|_\infty .
\]
\end{proof}

Observe that the reduction of the error per step  of the $i$th iteration is bounded from above
by $\|H_i\|_\infty$ which are in turn bounded by $\sigma$ for $i=1,2$  and by $\tau$  for $i=3$. These constants are explicitly computable by means of
\eqref{eq:sigma} and \eqref{eq:sigmatau}.

The following result shows that 
the convergence can be accelerated if $X_0$ is a stochastic matrix, say $X_0=I$, as in the finite dimensional case \cite{meini}.

\begin{theorem}\label{thm:a2}
Assume that $A_{-1}\one>A_1\one$. 
For $i\in\{1,2,3\}$ define $X_{k+1}^{(i)}=F_i(X_k^{(i)})$, $k=0,1,2,\ldots$, where $X_0^{(i)}\ge 0$, $X^{(i)}_0\one=\one$ and $F_i(X)$ are given in  \eqref{eq:fixp}. Then  
\begin{enumerate}
\item the three sequences $\{X_k^{(i)}\}$ are well defined,
\item   $X_k^{(i)}\ge 0$, and $X^{(i)}_k\one=\one$,
\item for the error $\E_k^{(i)}=G-X_k^{(i)}$ we have the following
  property: $\E_k^{(i)}\one=0$, and for any other eigenvector $w\ne \one$ of $G$ such that $Gw=\lambda w$, the sequence $w_k^{(i)}=\E_k^{(i)}w$ satisfies
  $\|w_{k+1}^{(i)}\|_\infty\le \|H_i(\lambda)\|_\infty
  \|w_{k}^{(i)}\|_\infty$, for $i=1,2$ where
  $H_1(\lambda)=(|\lambda|+1)A_1+A_0$,
  $H_2(\lambda)=(|\lambda|+1)(I-A_0)^{-1}A_1$. Moreover,
  $\|w_{k+1}^{(3)}\|_\infty\le
  |\lambda|\cdot\|(I-A_0-A_1X_k)^{-1}A_1\|_\infty\|w_k^{(3)}\|_\infty$,
  and $\limsup_k \frac{
    \|w_{k+1}^{(3)}\|_\infty}{\|w_k^{(3)}\|_\infty}\le |\lambda|$.
\end{enumerate}
\end{theorem}
\begin{proof}
We show that if $X\ge0$ and $X\one=\one$ then $F_i(X)\ge 0$ and
$F_i(X)\one=\one$. For $F_1$ this property can be easily checked. For
$F_2$, since $X\ge 0$ and $(I-A_0)^{-1}\ge 0$ then $F_2(X)\ge
0$. Moreover $F_2(X)\one= (I-A_0)^{-1}(A_{-1}+A_1)\one=
(I-A_0)^{-1}(I-A_0)\one=\one$. Concerning $F_3$, since $X\one=\one$ then
$\|A_0+A_1X\|_\infty=\|A_0+A_1G\|_\infty$ so that in light of Lemma
\ref{lem:2} the matrix $I-A_0-A_1X$ is invertible and has nonnegative
inverse. This implies that $F_3(X)\ge 0$. Moreover, since $X\one=\one$
then $(I-A_0-A_1X-A_{-1})\one=0$ so that
$F_3(X)\one=(I-A_0-A_1X)^{-1}A_{-1}\one=\one$.  From \eqref{eq:F1} we
obtain $w_{k+1}^{(1)}=(\lambda A_1 +A_1X_k+A_0) w_k^{(1)}$ so that
$\|w_{k+1}^{(1)}\|_\infty\le \| \lambda A_1 +A_1X_k+A_0\|_\infty
\|w_k^{(1)}\|_\infty$. On the other hand $\| \lambda A_1
+A_1X_k+A_0\|_\infty \le \| |\lambda|A_1+A_1X_k+A_0\|_\infty =
\|(|\lambda|A_1+A_1X_k+A_0) \one\|_\infty= \|(|\lambda|A_1+A_1+A_0)
\one\|_\infty =\| H_1(\lambda)\|_\infty$. Similarly, we proceed with
$F_2$ relying on \eqref{eq:F2}. Concerning $F_3$, from \eqref{eq:F3}
we have $w_{k+1}^{(3)}=\lambda(I-A_0-A_1X_k)^{-1}A_1w_k^{(3)}$, whence
$\|w_{k+1}^{(3)}\|_\infty\le |\lambda|
\|(I-A_0-A_1X_k)^{-1}A_1\|_\infty\|w_k^{(3)}\|_\infty$. Since
$A_1\one<A_{-1}\one$ then $\|(I-A_0-A_1X_k)^{-1}A_1\|_\infty
=\|(I-A_0-A_1X_k)^{-1}A_1\one\|_\infty\le
\|(I-A_0-A_1X_k)^{-1}A_{-1}\one\|_\infty=1$.  Taking the limsup for
$k\to\infty$ we obtain $\limsup_k \|(I-A_0-A_1X_k)^{-1}A_1\|_\infty
\le
\|(I-A_0-A_1G)^{-1}A_{-1}\|_\infty = \|G\|_\infty=1$. This completes
the proof.
\end{proof}

Observe that the condition $\lim_k \|G-X_k\|_\infty=0$ implies that 
$\|w_{k+1}^{(3)}\|_\infty/ \|w_{k}^{(3)}\|_\infty
\le |\lambda|\lim_k\|(I-A_0-A_1X_k)^{-1}A_1\|_\infty=|\lambda|\cdot\|H_3\|_\infty\le |\lambda|\cdot\tau$.

According to the above theorem, the sequences generated by the three
functional iterations with $X_0$ stochastic, converge faster than the
corresponding sequences obtained with $X_0=0$. For the functions $F_1$
and $F_2$, this follows since
$\|H_1(\lambda)\|_\infty\le\|H_1\|_\infty$ and
$\|H_2(\lambda)\|_\infty\le\|H_2\|_\infty$. Similarly, we may proceed for the function $F_3$.
The convergence of the sequence is faster the smaller
$\sup\{|\lambda|:\,\lambda\ne 1,~ \lambda \hbox{ ei}\hbox{gen}\hbox{va}\hbox{lue of }G\}$.

\subsection{Implementation issues}
Let $g(z)$ be the solution of minimum modulus of \eqref{ag} and
consider the sequences generated by the fixed point iterations
\eqref{eq:fixp} obtained starting with $X_0=T(g)+C$, where $C$ is any
correction. Denote by $X_k$ any one of these three sequences so that
we have $X_k=T(g)+E_k$, $E_0=C$, where $E_k$ is the correction
part. In this section we aim to explicit the equation which relates
$E_{k+1}$ to $E_k$.

Consider the first iteration $X_{k+1}=A_1X_k^2+A_0X_k+A_{-1}$ and
denote by $F$ the correction matrix such that
$T(g)+F=A_1T(g)^2+A_0T(g)+A_{-1}$. Subtracting the latter equation
from the former and performing formal manipulations yields
\begin{equation}\label{eq:F1.1}
E_{k+1}=F+(A_1E_k+S)E_k+A_1E_kT(g),\quad S=A_0+A_1T(g).
\end{equation}
This equation provides a more efficient way to implement the first
iteration since it involves multiplications of QT matrices and
correction matrices and avoids the multiplication of QT matrices
having a nonzero symbol. In this version, the precomputation of $g$,
$S$ and $F$ is needed.

Consider the second iteration
$X_{k+1}=(I-A_0)^{-1}(A_1X_k^2+A_{-1})$. Subtract it from the equation
$T(g)=(I-A_0)^{-1}(A_1T(g)^2+A_{-1})-(I-A_0)^{-1}F$ and performing
formal manipulations yields
\begin{equation}\label{eq:F2.1}
\begin{aligned}
&E_{k+1}=\widehat S((T(g)+E_k)E_k+E_kT(g))+\widetilde S,\\
& \widehat S=(I-A_0)^{-1}A_1,~\widetilde S=(I-A_0)^{-1}F.
\end{aligned}\end{equation}
Also in this case the precomputation of $T(g)$, $\widehat S$, $F$ and $\widetilde S$ is needed.

Concerning the third iteration and proceeding similarly we arrive at the recursion
\begin{equation}\label{eq:F3.1}
\begin{aligned}
&E_{k+1}=\widehat VE_k(I-A_1(T(g)+E_k)-A_0)^{-1}A_{-1} +\widetilde V,\\
& \widehat V=(I-A_1T(g)-A_0)^{-1}A_1,~\widetilde V=(I-A_1T(g)-A_0)^{-1}F.
\end{aligned}\end{equation}
Also in this case the precomputation of $T(g)$, $\widehat V$, and $\widetilde V$ is needed.
However, at each step the inverse of a QT matrix must be computed.

The iterations \eqref{eq:F1.1}--\eqref{eq:F3.1} can be started with
$X_0=T(g)$, that is, $E_0=0$. Alternatively, they can be started with $X_0=T(g)+ve_1^T$,
where $e_1^T=(1,0,\ldots)$ and $v$ is chosen in such a way that $X_0$ is
row-stochastic so that convergence is faster. This can be accomplished
by setting $E_0=ve_1^T$.

\section{Newton's method}\label{sec:newton}
Rewrite equation $A_1X^2+A_0X+A_{-1}=X$ as
\begin{equation}\label{eq:mateqx}
L(X)=0,\quad L(X):=A_1X^2+(A_0-I)X+A_{-1}.
\end{equation}
Newton's method applied to equation \eqref{eq:mateqx} generates the sequence
\begin{equation}\label{eq:newt}
X_{k+1}=X_k-Z_k,~~k=0,1,\ldots,
\end{equation}
where the matrix $Z_k$ solves the equation $L'(Z_k)=L(X_k)$, and $L'(H)=A_1XH+A_1HX+(A_0-I)H$ is the Fr\'echet derivative of $L(X)$ at $X$ applied to the matrix $H$. More specifically, $Z_k$ solves the following Sylvester equation 
\begin{equation}\label{eq:sylv}
(A_1X_k+A_0-I)Z_k+A_1Z_kX_k= L(X_k).
\end{equation}
Observe that we may write
\begin{equation}\label{eq:hk}
Z_k=-\sum_{i=0}^\infty S_k^i(I-A_0-A_1X_k)^{-1}L(X_k)X_k^i,\quad S_k=(I-A_0-A_1X_k)^{-1}A_1,
\end{equation}
provided that $\|(I-A_0-A_1X_k)^{-1}\|_\infty$ is bounded from above,  $\|S_k\|_\infty<1$ and $\|X_k\|_\infty\le 1$ so that we have
\begin{equation}\label{eq:sk}
\|Z_k\|_\infty\le \frac{\|L(X_k)\|_\infty
\| (I-A_0-A_1X_k)^{-1} \|_\infty }{1-\|S_k\|_\infty}.
\end{equation}
Moreover, we have
\begin{equation}\label{eq:ah2}
L(X_{k+1})=A_1Z_k^2.
\end{equation}
The latter equality can be proved by observing that in general, for $Y=X-H$, we have
$L(Y)=A_1(X-H)^2+(A_0-I)(X-H)+A_{-1}=L(X)-L'(H)+A_1H^2$ so that, if $H$ is such that $L'(H)=L(X)$ as in a Newton step, then
$L(Y)=A_1H^2$.

Another useful property is the following. 
Equation \eqref{eq:newt} can be rewritten as $Z_k=\E_{k+1}-\E_k$, where $\E_k=G-X_k$, and $L(X_k)$ can be rewritten as $L(X_k)=L(X_k)-L(G)=-A_1(\E_kG+X_k\E_k)-(A_0-I)\E_k$.
Replace these two representations for $Z_k$ and $L(X_k)$ in \eqref{eq:sylv}, and get
\[
(I-A_0-A_1X_k)\E_{k+1}-A_1\E_{k+1}X_k=A_1\E_k^2.
\]
By following the same arguments used to arrive at \eqref{eq:hk},
 we may rewrite the above equation as
\[
\E_{k+1}-(I-A_0-A_1X_k)^{-1}A_1\E_{k+1}X_k=(I-A_0-A_1X_k)^{-1}A_1\E_k^2
\]
and get
\begin{equation}\label{eq:ekp1}
\E_{k+1}=\sum_{i=0}^\infty S_k^{i+1} \E_k^2 \, X_k^i.
\end{equation}

The following result extends to QT matrix coefficients the convergence results valid in the finite dimensional case \cite{latoucheN}:

\begin{theorem}\label{thm:newt}
Assume that $A_{-1}\one>A_1\one$. Let $X_k$, $k=0,1,\ldots,$ be
 the sequence generated by \eqref{eq:newt} and \eqref{eq:sylv}
 starting with $X_0=0$. Then,  for any $k=0,1,2,\ldots$,
\begin{enumerate}
\item equation \eqref{eq:sylv} has a solution $Z_k$ such that $\|Z_k\|_\infty\le\beta$, where
  \[ \beta=\frac{2\|W^{-1}\|_\infty}{1-\|W^{-1}A_1\|_\infty},
  \]
  for $W=I-A_0-A_1G$, so that $X_{k+1}$ is well defined; 
\item $Z_k\le 0$, $L(X_{k+1})\ge 0$ and $0\le X_k\le X_{k+1}\le G$;
\item  $\lim_{k\to\infty}(\E_k)_{i,j}=0$ for any $i,j\ge 1$, where $\E_k=G-X_k$;
\item  $\|\E_{k+1}\|_\infty\le \frac{\tau}{1-\tau}\|\E_k\|^2_\infty$, $\|Z_{k+1}\|_\infty\le \frac{\tau}{1-\tau}\|Z_k\|^2_\infty$, where $\tau=\min\{\gamma,\sigma\}$, with $\gamma$ and $\sigma$  defined in \eqref{eq:gamma} and \eqref{eq:sigma}, respectively.
\end{enumerate}
\end{theorem}

\begin{proof}
We prove properties 1 and 2 by induction on $k$. For $k=0$ we have
$X_0=0$, $Z_0=-(I-A_0)^{-1}A_{-1}\le 0$, $L(X_1)=A_{1}Z_0^2\ge 0$,
$X_1=-Z_0\ge X_0$ and $X_1=(I-A_0)^{-1}A_{-1}\le G$. Moreover, clearly
$\|Z_0\|_\infty\le\beta$. For the inductive step, assume that
properties 1 and 2 are valid for $k$ and prove them for
$k+1$. We show that $\|Z_{k+1}\|\le \beta$. Consider \eqref{eq:sk}.
Since by induction $L(X_{k+1})\ge0$ then $\|L(X_{k+1})\|_\infty=\|L(X_{k+1})\one\|_\infty
\le \|(A_1X_{k+1}^2+A_0X_{k+1}+A_{-1})\one\|_\infty+\|X_{k+1}\one\|_\infty$. Moreover, since 
$X_{k+1}\le G$ then $X_{k+1}\one\le G\one$ so that $\|L(X_{k+1})\|_\infty\le 2$.
From $X_{k+1}\le G$ it follows also $A_0+A_1X_{k+1}\le A_0+A_1G$ so that
$(I-A_0-A_1X_{k+1})^{-1}\le(I-A_0-A_1G)^{-1}=W^{-1}$, whence
$\|(I-A_0-A_1X_{k+1})^{-1}\|_\infty\le \|W^{-1}\|_\infty$. Similarly,
$\|S_k\|_\infty\le\|(I-A_0-A_1G)^{-1}A_1\|_\infty=\|W^{-1}A_1\|_\infty<1$
in view of Lemma
\ref{lem:2}. From \eqref{eq:hk} and \eqref{eq:sk} we get
$\|Z_k\|_\infty\le\beta$.  The property $Z_{k+1}\le 0$ follows from
\eqref{eq:hk} since $S_{k+1} \ge 0$, $(I-A_0-A_1X_{k+1})^{-1}\ge 0$,
and $L(X_{k+1})\ge 0$ by the inductive assumption. The inequality
$L(X_{k+2})\ge 0$ follows from \eqref{eq:ah2} since $Z_{k+1}\le 0$.
Consequently $X_{k+2}=X_{k+1}-Z_{k+1}\ge X_{k+1}$. The property
$X_{k+2}\le G$ follows from \eqref{eq:ekp1} since $\E_{k+1}\ge 0$,
$S_{k+1}\ge 0$ and $X_{k+1}\ge 0$.  Given $i,j$ consider the sequence
$(\E_k)_{i,j}$ for $k=0,1,\ldots$. This sequence is non-increasing and
bounded from below by $0$ therefore it has a limit. The value of the
limit cannot be positive since $G$ is the minimal nonnegative solution
to the matrix equation.  From the representation \eqref{eq:ekp1} of
$\E_{k+1}$, since $\|X_k\|_\infty\le\|G\|_\infty\le 1$, and
$\|S_k\|_\infty\le\|(I-A_0-A_1G)^{-1}A_1\|_\infty\le\tau<1$ in view of
\eqref{eq:sigmatau}, we deduce that
\[
\|\E_{k+1}\|_\infty\le \frac\tau{1-\tau}\|\E_k\|_\infty^2.
\]
Similarly we can do for $\|Z_{k+1}\|_\infty$.
\end{proof}

\section{Perturbation results}\label{sec:perturb}
For the case where the coefficient matrices are finite, Higham and Kim \cite{HK} derived a condition number $\Psi(X)$ for a solvent $X$ of a general quadratic matrix equation of the kind \eqref{eq:G}
namely,
\[
\Psi(X)=\|P^{-1}[\alpha (X^2)^T\otimes I_n, \beta X^T\otimes I_n, \gamma I_n^2]\|_2/\|X\|_F,
\]
where $P=I_n\otimes A_1X+X^T\otimes A_1+I_n\otimes (A_0-I)$ and $\alpha, \beta, \gamma$ are nonnegative parameters.

 However, when the coefficient matrices are semi-infinite, there are cases where $P^{-1}$ does not exist or $\|X\|_F=\infty$, so the definition of $\Psi(X)$ does not apply.
 In this section, we take into account the structure of the coefficient matrices and derive a structured condition number for the minimal nonnegative solution of equations \eqref{eq:G} and \eqref{eq:R}. Without loss of generality we consider only equation \eqref{eq:G}.

Consider the perturbed matrix equation obtained from \eqref{eq:G} by replacing the coefficients $A_i$ by $A_i+\Delta_{A_i}$
where $\Delta_{ A_i}=T(\delta_i)+E_{\delta_i}\in \mathcal{QT},$  $A_i+\Delta_{ A_i}\geq 0$, for  $i=-1,0,1$,  and  $(A_1+\Delta_{ A_1}+A_0+\Delta_{ A_0}+A_{-1}+\Delta_{ A_{-1}})\bf{1}=\bf{1}$. Denote $X+\Delta_X$ a solution of the perturbed equation so that we may write

\begin{equation}\label{qme-perturbed}
\begin{split}
(A_{1}+\Delta_{{A}_{1}})(X+\Delta_ X)^2+&(A_0+\Delta_{ A_0})(X+\Delta_ X)
+
A_{-1}+\Delta_{ A_{-1}}=X+\Delta_ X.
\end{split}\end{equation}

The analysis is separated into two parts, that is, the the analysis of the structured condition number of the Toeplitz part and the analysis of the condition number of the whole matrix. 

\subsection{Toeplitz part}\label{ssec:toep}
In this section we provide a perturbation result for the function $g(z)$ which is the solution of minimum modulus of the scalar equation \eqref{ag}.

For the sake of notational simplicity, we omit the variable $z$ from the symbols, say, we write $g$ in place of $g(z)$ and $a_i$ in place of $a_i(z)$.

Under the assumption that the matrix coefficients $A_i+\Delta_{A_i}$ of equation \eqref{qme-perturbed} still satisfy the condition $A_i+\Delta_{A_i}\ge 0$, $\sum_{i=-1}^1(A_i+\Delta_{A_i})\one=\one$,
for Theorem \ref{thm:A} 
the minimal nonnegative solution of \eqref{qme-perturbed} can be written as $G+\Delta_ G$ , where $\Delta_ G=T(\delta_{g})+E_{\delta_{g}}\in \mathcal{QT}$ and $g+\delta_g$ is the solution of minimum modulus of the equation 
\[
a_{-1}+\delta_{-1}+(a_0+\delta_0)\mu+(a_1+\delta_1)\mu^2=\mu.
\]
Taking the difference of the above equation with  \eqref{ag}, where we set $\mu=g+\delta_g$ and $\lambda=g$, we obtain
\[
\delta_{-1}+\delta_0(g+\delta_g) +\delta_1(g+\delta_g)^2+(a_0-1)\delta_g+a_1((g+\delta_g)^2-g^2)=0.
\] 
Whence, neglecting higher order terms in the perturbations we get
\[
\delta_{-1}+\delta_0g +\delta_1g^2+(a_0-1)\delta_ g+2a_1g\delta_g\doteq 0,
\]
where $\doteq$ means equality up to higher order terms with respect to the perturbations. This yields
\begin{equation}\label{Teo-Per2}
\delta_{g}\z\doteq\frac{\delta_1\z g\z^2+\delta_0\z g\z+\delta_{-1}\z}{1-2a_1\z g\z-a_0\z}.
\end{equation}

Note that $g(z)$, $a_0(z)$ and  $a_1(z)$ have nonnegative coefficients so that 
$\|g\|_{w}=\sum_{i\in \mathbb{Z}}g_i=g(1)=1$, and  
$|2a_1(z)g(z)+a_0(z)|\le 2a_1(1)g(1)+a_0(1)=2a_1(1)+a_0(1)=1-a_{-1}(1)+a_1(1)<1$, due to \eqref{eq:assumption1}. This way, we have    
$\|(1-2a_1g-a_0)^{-1}\|_{w}=(1-2a_1(1)g(1)-a_0(1))^{-1}=(a_{-1}(1)-a_1(1))^{-1}$. Whence, from (\ref{Teo-Per2}), we obtain
\[
\|\delta_{g}\|_{w}\dotle (a_{-1}(1)-a_1(1))^{-1} \|\delta_1g^2+\delta_0g+\delta_{-1}\|_{w},
\]
where $\dotle$ means inequality up to higher order terms with respect to the perturbations. Therefore we arrive at the bound

\begin{equation}\label{Teo-bound}
\|\delta_{g}\|_{w}\dotle \frac{1}{a_{-1}(1)-a_1(1)}\big(\|\delta_1\|_{w}+\|\delta_0\|_{w}+\|\delta_{-1}\|_{w}\big).
\end{equation}

If we measure the  perturbations by
\[
\epsilon={\rm max}\Big\{\frac{\|\delta_i\|_w}{\|a_i\|_w},\ i=-1,0,1\Big\},
\]
 we have $\|\delta_i\|_w\le\epsilon\|a_i\|_w$, moreover, since $\|a_{-1}\|_w+\|a_0\|_w+\|a_1\|_w=a_{-1}(1)+a_0(1)+a_1(1)=1$, the relative variation of the symbol is bounded by
\begin{equation}\label{Toe-cond}
\frac{\|\delta_{g}\|_w}{\|g\|_w}=\|\delta_{g}\|_w\leq\frac{1} {a_{-1}(1)-a_1(1)}\epsilon +O(\epsilon^2).
\end{equation}

It follows from \eqref{Toe-cond} that $\hbox{cond}_{T(g)}:=\frac{1} {a_{-1}(1)-a_1(1)}$ is an upper bound to the condition number of the Toeplitz part of $G$.

\subsection{Whole matrix}\label{ssec:corr}

Expanding \eqref{qme-perturbed}, omitting the second and higher order terms in the perturbations, and setting $X=G$ lead to
\begin{equation}\label{eq:cp}
(I-A_1G-A_0)\Delta_ G-A_1\Delta_ G G \doteq (\Delta_{ A_1}G^2+\Delta_{ A_0}G+\Delta_{ A_{-1}}).
\end{equation}

Now we prove some properties that will be useful to estimate the condition number of the whole solution.

Let $\Delta_A:=\Delta_{ A_1}G^2+\Delta_{ A_0}G+\Delta_{ A_{-1}}$.
According to (\ref{eq:wh}), equation \eqref{eq:cp} can be written as
\begin{equation}\label{W2}
F(\Delta_G)\doteq W^{-1}\Delta_A
\end{equation}
where $F:\mathcal {QT}\rightarrow \mathcal {QT}$ is the map defined by
\[F(Y)=Y-(W^{-1}RW)YG.
\]

Now, we prove that the map $F(Y)$ is invertible in $\mathcal L_\infty$, that is, $F^{-1}$ has bounded infinity norm. 
By Lemma \ref{lem:1}, we have $\|W^{-1}RW\|_\infty\le\gamma<1$ and $\|G\|_\infty= 1$ so that the series  
$\sum_{k=0}^\infty (W^{-1}RW)^kVG^k$ is convergent for any $V\in\mathcal L_\infty$. Therefore, if $V=F(Y)$ then  $Y=\sum_{k=0}^\infty (W^{-1}RW)^kVG^k$.  Thus, we get
\[
\Delta_G=F^{-1}(W^{-1}\Delta_A)=\sum_{k=0}^\infty (W^{-1}RW)^k(W^{-1}\Delta_A)G^k.
\]
Since $\sum_{k=0}^\infty \|W^{-1}RW\|_\infty^k=1/(1-\|W^{-1}RW\|_\infty)$,  taking norms in the above expression and applying Lemma \ref{lem:1} yields
\begin{equation}\label{eq:bound1}
\|\Delta_G\|_\infty\le\frac{\|W^{-1}\|_\infty}{1-\|W^{-1}RW\|_\infty}\|\Delta_A\|_\infty\le\frac1{\theta(1-\gamma)}\|\Delta_A\|_\infty.
\end{equation}

Whence we conclude with the following 

\begin{theorem}If $A_{-1}\one>A_1\one$, then for the perturbation $\Delta_G$ we have
\begin{equation}\label{eq:2.10}\begin{aligned}
\|\Delta_G\|_\infty\le& 
\frac{\|W^{-1}\|_\infty}{1-\|W^{-1}RW\|_\infty}\|\Delta_A\|_\infty\\
\le& \frac1{\theta(1-\gamma)}\|\Delta_A\|_\infty\le \frac1{\theta(1-\gamma)} (\|\Delta_{A_{-1}}\|_\infty+\|\Delta_{A_0}\|_\infty+\|\Delta_{A_1}\|_\infty),
\end{aligned}\end{equation}
where $\theta$ and $\gamma$ are defined in \eqref{eq:gamma} and $\Delta_A=\Delta_{A_1}G^2+\Delta_{A_0}G+\Delta_{A_{-1}}$. 
\end{theorem}

From the above result it turns out that $\|W^{-1}\|_\infty/(1-\|W^{-1}RW\|_\infty)$ is an
estimate of the conditioning of the problem, while $1/(\theta(1-\gamma))$ provides an upper bound.
Since $1-a_{1}(1)/a_{-1}(1)=(a_{-1}(1)-a_1(1))/a_{-1}(1)$, 
we may rewrite the upper bound to the conditioning in the following form which is closer  to the expression obtained for the Toeplitz part of $G$ in Section \ref{ssec:toep}. 
\[
\frac1{\theta(1-\gamma)}= \max\left(\frac {a_{-1}(1)}{a_{-1}(1)-a_1(1)}, \frac {b_{-1}(1)}{b_{-1}(1)-b_1(1)}\right)\frac1{\min(a_{-1}(1),b_{-1}(1))}.
\]

It is interesting to observe that if $a_{-1}(1)\le b_{-1}(1)$ and $\frac{a_{-1}(1)}{b_{-1}(1)}\le\frac{a_1(1)}{b_1(1)}$, which in turn is verified if $a_{-1}(1)\le b_{-1}(1)$ and $b_1(1)\le a_1(1)$, then
\[
\frac1{\theta(1-\gamma)}=\frac1{a_{-1}(1)-a_1(1)},
\]
that is, the conditioning of the Toeplitz part coincides with the conditioning of the whole problem.

It is interesting to point out that, since $RW=A_1$ (see equation \eqref{eq:wh}), the estimate of the condition number  $\|W^{-1}\|_\infty/(1-\|W^{-1}RW\|_\infty)$, appears in the uniform bound $\beta$ to the norm of the Newton correction $Z_k$ introduced in Theorem \ref{thm:newt}. Consequently, if the quadratic matrix equation is well conditioned, the uniform bound to the norm of $Z_k$ is smaller.

\subsection{A simple example}\label{sec:condexp}

This example is taken from Example 6.2 in \cite{Motyer-Taylor},  where a continuous time Markov process modeling a two-node Jackson network is considered. Here, the matrices are modified by means of the uniformization technique \cite{lr:book} in order to represent a discrete-time model. 
In details,
\[
A_{-1}=\alpha\left(
         \begin{array}{cccc}
           (1-q)\mu_2 & q\mu_2 &   &   \\
             & (1-q)\mu_2 & q\mu_2 &   \\
             &   &   \ddots & \ddots  \\
         \end{array}
       \right),
A_1=\alpha\left(
      \begin{array}{ccc}
        \lambda_2 &  &   \\
        p\mu_1 & \lambda_2 & \\
         & \ddots & \ddots \\
      \end{array}
    \right),
\]

\[
A_0=\alpha\left(
      \begin{array}{cccc}
        -(\lambda_1+\lambda_2+\mu_2) & \lambda_1 &  &  \\
        (1-p)\mu_1 & -(\lambda_1+\lambda_2+\mu_1+\mu_2) & \lambda_1 & \\
          & \ddots &\ddots & \ddots \\
      \end{array}
    \right) + I,
\]
where the parameters $\lambda_1, \lambda_2, \mu_1, \mu_2, p, q$ are chosen as in Table \ref{tab:mt} and $\alpha = (\lambda_1+\lambda_2+\mu_1+\mu_2)^{-1}$.
Each case denotes one instance of the two-node Jackson network, formed by two servers and two queues, where
customers arrive at nodes 1 and 2 according to independent Poisson processes with rates $\lambda_1$ 
and $\lambda_2$, respectively. Customers are served according to a first-come-first-served
discipline, service times at nodes 1 and 2 are independent and exponentially
distributed with means $1/\mu_1$, $1/\mu_2$.
After completing service at node 1,
customers enter node 2 with probability $p$ or leave the system with probability $1 - p$, where
$0 < p < 1$.  After completing service at node 2, customers enter node 1
with probability $q$ or leave the system with probability $1 - q$, where $0 < q < 1$.

In \cite{Motyer-Taylor}, 10 cases defined by the parameters given in Table
\ref{tab:mt} are analyzed.
It can be easily  seen that the condition $A_{-1}\one> A_{1}\one$  holds for cases 1, 3, 4, 5, 7, 8, 9, while the same condition holds in the cases 2, 6 and 10 for the flipped problems where phases and levels are exchanged.

\begin{table}\label{tab:mt}\begin{center}
{\footnotesize\begin{tabular}{ccccccc}
Case&$\lambda_1$ &$\lambda_2$ &$\mu_1$ &$\mu_2$ &$p$&$q$\\ \hline
1&    1 &0 &1.5 &2 &1 &0\\
2&    1& 0& 2& 1.5& 1& 0\\
3&    0& 1& 1.5& 2& 0& 1\\
4&    0& 1& 2& 1.5& 0& 1\\
5&    1& 1& 2& 2& 0.1& 0.8\\
6&    1& 1& 2& 2& 0.8& 0.1\\
7&    1& 1& 2& 2& 0.4& 0.4\\
8&    1& 1& 10& 10& 0.5& 0.5\\
9&    1& 5& 10& 15& 0.4& 0.9\\
10&   5& 1& 15& 10& 0.9& 0.4\\
\end{tabular}\caption{Parameters defining the matrices $A_{-1},A_0,A_1$ in the 2-node Jackson network of \cite{Motyer-Taylor}}}\end{center}
\end{table}

For the flipped problem,   the coefficient matrices are
\[
A_{-1}=\alpha\left(
         \begin{array}{cccc}
           (1-p)\mu_1 & p\mu_1 &   &   \\
             & (1-p)\mu_1 & p\mu_1 &   \\
             &   &   \ddots & \ddots  \\
         \end{array}
       \right),
A_1=\alpha\left(
      \begin{array}{ccc}
      \lambda_1 &  &   \\
      q\mu_2& \lambda_1 & \\
      & \ddots & \ddots \\
      \end{array}
    \right),
\]

\[
A_0=\alpha\left(
      \begin{array}{cccc}
        -(\lambda_1+\lambda_2+\mu_1) & \lambda_2 &  &  \\
        (1-q)\mu_2 & -(\lambda_1+\lambda_2+\mu_1+\mu_2) & \lambda_2 & \\
          & \ddots &\ddots & \ddots \\
      \end{array}
    \right) + I.
\]
If we perturb the parameters $\lambda_1, \lambda_2, \mu_1, \mu_2$ by

\[\begin{aligned}
&\tilde{\lambda}_i=\lambda_i(1+\epsilon_i^\lambda)\quad \tilde{\mu}_i=\mu_i(1+\epsilon_i^\mu),\quad
  i=1,2,\\
  & \epsilon_i^\lambda,\epsilon_i^\mu\in[10^{-8},2\cdot 10^{-8}],
\end{aligned}
\]
where the perturbations are randomly chosen,
then we get the perturbed matrices $A_{-1}+\Delta_{ A_{-1}}$, $A_0+\Delta_{ A_0}$ and $A_1+\Delta_{ A_1}$.

Note that $a_{-1}(1)=b_{-1}(1)$, $a_{-1}(1)-a_1(1)<b_{-1}(1)-b_1(1)$
holds true for both the original and the flipped problems, it follows
$\frac{1}{\theta(1-\gamma)}=\frac{1}{a_{-1}(1)-a_{1}(1)}$, that is,
the conditioning of the Toeplitz part coincides with the conditioning
of the whole problem.

We denote by ``Conditioning'' the upper bound on the condition number
for the minimal nonnegative solution $G$, that is,
$\frac{1}{\theta(1-\gamma)}$. Moreover, we denote by $\delta_g$-bound
and $\Delta_G$-bound, respectively, the perturbation bound
\eqref{Teo-bound} on $\|\delta_g\|_w$ and the bound \eqref{eq:2.10} on
$\|\Delta_G\|_{\infty}$. In Table \ref{tab:con}, we report the upper
bound on the condition number for the minimal nonnegative solution $G$
of equation \eqref{eq:G}, and we compare the perturbation bound
\eqref{Teo-bound} on the Toeplitz part of $G$ and the bound on the
solution $G$ with the corresponding perturbation errors.

It can be seen from Table \ref{tab:con} that the upper bound
$\frac{1}{\theta(1-\gamma)}$ can serve as a very good estimate of the
conditioning of the problem.  Inequalities \eqref{Teo-bound} and
\eqref{eq:2.10}
provide very sharp and revealing perturbation bounds
to the Toeplitz part and to the solution $G$ with respect to small
perturbations on the coefficients.

\begin{table}
{\footnotesize
\begin{tabular*}{\textwidth}{c@{\extracolsep{\fill}}cccccccc}
\toprule
 Problem&  Conditioning  & $\|\delta_g\|_w$& $\delta_g$-bound & $\|\Delta_G\|_{\infty}$ & $\Delta_G$-bound\\
\midrule
1                     &9.0000 & 3.0601e-09 & 1.3211e-08 & 2.3583e-09 & 1.9817e-08\\
\hspace*{0.5em}2$^*$  &4.5000 & 1.5565e-09 & 4.8528e-09 & 1.5649e-09 & 4.8528e-09\\
3                     &4.5000 & 2.9937e-09 & 9.3015e-09 & 4.2850e-09 & 2.0256e-08\\
4                     &9.0000 & 4.7978e-09 & 2.0256e-08 & 1.8119e-08 & 2.5845e-09\\
5                     &7.5000 & 5.9099e-09 & 1.2456e-09 & 1.1106e-09 & 5.9334e-09\\
\hspace*{0.5em}$6^*$  &7.5000 & 4.4110e-10 & 8.2441e-09 & 3.8574e-10 & 8.3809e-09\\
7                     &30.0000 & 5.4766e-09& 6.8300e-08 & 5.4809e-09 & 8.1645e-08\\
8                     &5.5000 & 4.9898e-10 & 4.0333e-09 & 5.9838e-10 & 4.9549e-09\\
9                     &5.1667 & 7.7413e-10 & 2.8017e-09 & 7.7413e-10 & 3.1621e-09\\
\ $10^*$              &5.1667 & 6.3154e-10 & 3.8820e-09 & 5.3199e-10 & 4.4169e-09\\
\bottomrule
\end{tabular*}
}
\caption{Conditioning of the matrix equation for the 2-node Jackson network of \cite{Motyer-Taylor}: actual perturbations in the solution and in the Toeplitz part and related upper bounds.}
\label{tab:con}
\end{table}

\section{Computational issues and numerical experiments}\label{sec:exp}

Observe that, since the class $\mathcal {QT}$ is an algebra, then all
the matrices $X_k^{(i)}$ generated by the fixed point iterations of
Section \ref{sec:fxp} belong to $\mathcal{QT}$ and each fixed point
iteration can be easily implemented in Matlab relying on the
CQT-Toolbox of \cite{cqttoolbox}. Concerning Newton iteration, a
crucial role is played by the solution of the Sylvester-like equation
\eqref{eq:sylv}.  In the case where the coefficients have finite size
$n$, this equation can be solved by the Bartels-Stewart algorithm
\cite{bs} in $O(n^3)$ arithmetic operations. The case of infinite size
coefficients is much more complicated. For quasi-Toeplitz matrices,
the problem has been analyzed in \cite{robol} by using rational Krylov
subspaces techniques where it is proved that $Z_k\in\mathcal{QT}$
under suitable assumptions that are satisfied by the condition
$A_{-1}\one>A_1\one$.  In the numerical experiments reported in this
section we have used the implementation of \cite{robol} to solve
\eqref{eq:sylv}.

We have implemented the three fixed point iterations analyzed in
Section \ref{sec:fxp} both in the standard versions \eqref{eq:fixp},
and in the version which separately computes the correction part, see
\eqref{eq:F1.1}, \eqref{eq:F2.1} and \eqref{eq:F3.1}. We also
implemented the computation of the coefficients of $g(z)$ relying on
the evaluation and interpolation strategy at the roots of 1 with
automatic handling of the number of interpolation points described in Algorithm \ref{alg:g}.

For each fixed point iteration, we have considered four different
starting approximations, namely, $X_0=0$, $X_0=I$, $X_0=T(g)$,
$X_0=T(g)+ve_1^T$ where $e_1=(1,0,\ldots)^T$ and $v$ is chosen so that $X_0$ is
row-stochastic.

Since there is not much difference in the performances of the versions
based on computing only the correction and the versions where the whole
matrix is computed, we report only the results concerning the latter versions.

We have compared the three fixed point iterations, with different
choices of $X_0$, to Newton's iteration and to the algorithm of cyclic
reduction (CR) analyzed in \cite{bmmr}, which, in the case of finite
matrices, is the method of choice commonly used in practice. For each
experiment, we report the number of iterations, and CPU time needed to
reach the bound $\|A_1X^2+(A_0-I)X+A_{-1}\|_\infty \le\epsilon $ to
the residual error where $\epsilon=5.0\cdot 10^{-14}$.  We have considered some
test problems modeling real world networks. More precisely, the
``Two-node Jackson network'' of Example 6.2 in \cite{Motyer-Taylor},
reported in Section \ref{sec:condexp}, and the model ``Assistance from an
idle server'' of \cite{Motyer-Taylor} with different choices of the parameters,
together with a general random walk in
the quarter plane where the assigned probabilities have been chosen in
such a way to have long queues in the system.

\subsection{Two-node Jackson network}\label{sec:2nJn}
The model that we consider has been described in Section
\ref{sec:condexp}. Among the 10 problems in the list of Table
\ref{tab:mt}, we report the results of Problem 7 which is the most
ill-conditioned in the list, together with the case obtained with
different values of the parameters $\lambda_i,\mu_i$, $p$ and $q$
which make the matrix $G$ numerically very large so that the
computational effort is substantially large.

Figure \ref{fig:MT7} concerns Problem 7 in the list of
Table \ref{tab:mt}.  The first three graphs report the residual errors
at each step for different values of $X_0$. The fourth graph compares
the residual errors of the three iterations for $X_0=T(g)+v e_1^T$. In Table \ref{tab:MT7} it is reported the CPU time in seconds
(left) together with the number of steps (right) required by the three
iterations to arrive at a residual error at most $5.0\cdot 10^{-14}$. The
computation of the symbol $g(z)$ is very inexpensive since the
coefficients $g_i$ are computed in $0.003$ seconds.  For this problem,
cyclic reduction provides the solution in just 8 steps and in 1.97
seconds, while Newton iteration requires 8 steps but takes a larger
amount of seconds, i.e., 256.8.

\pgfdeclareimage[width=5cm]{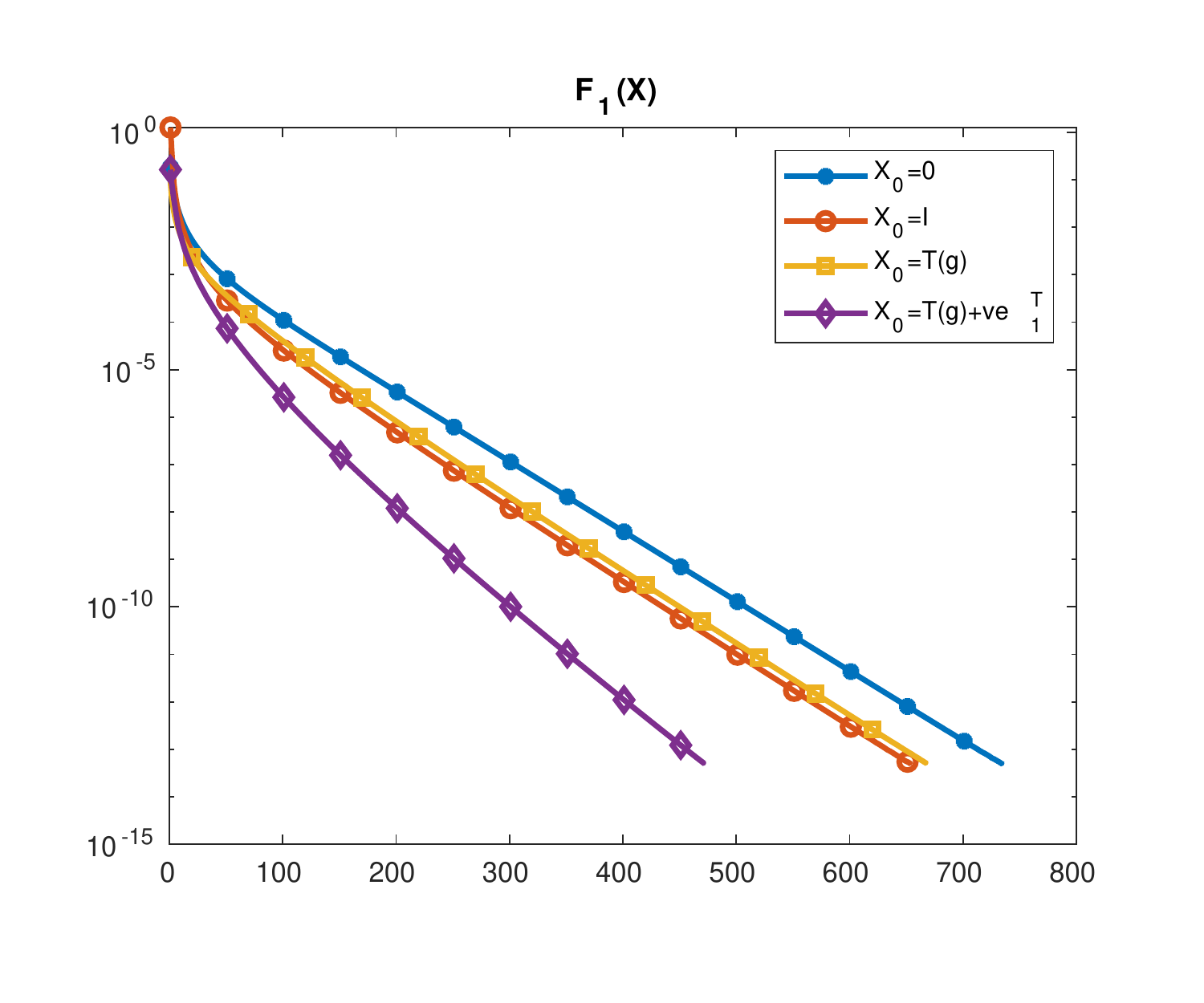}{MT7F1.pdf}  
\pgfdeclareimage[width=5cm]{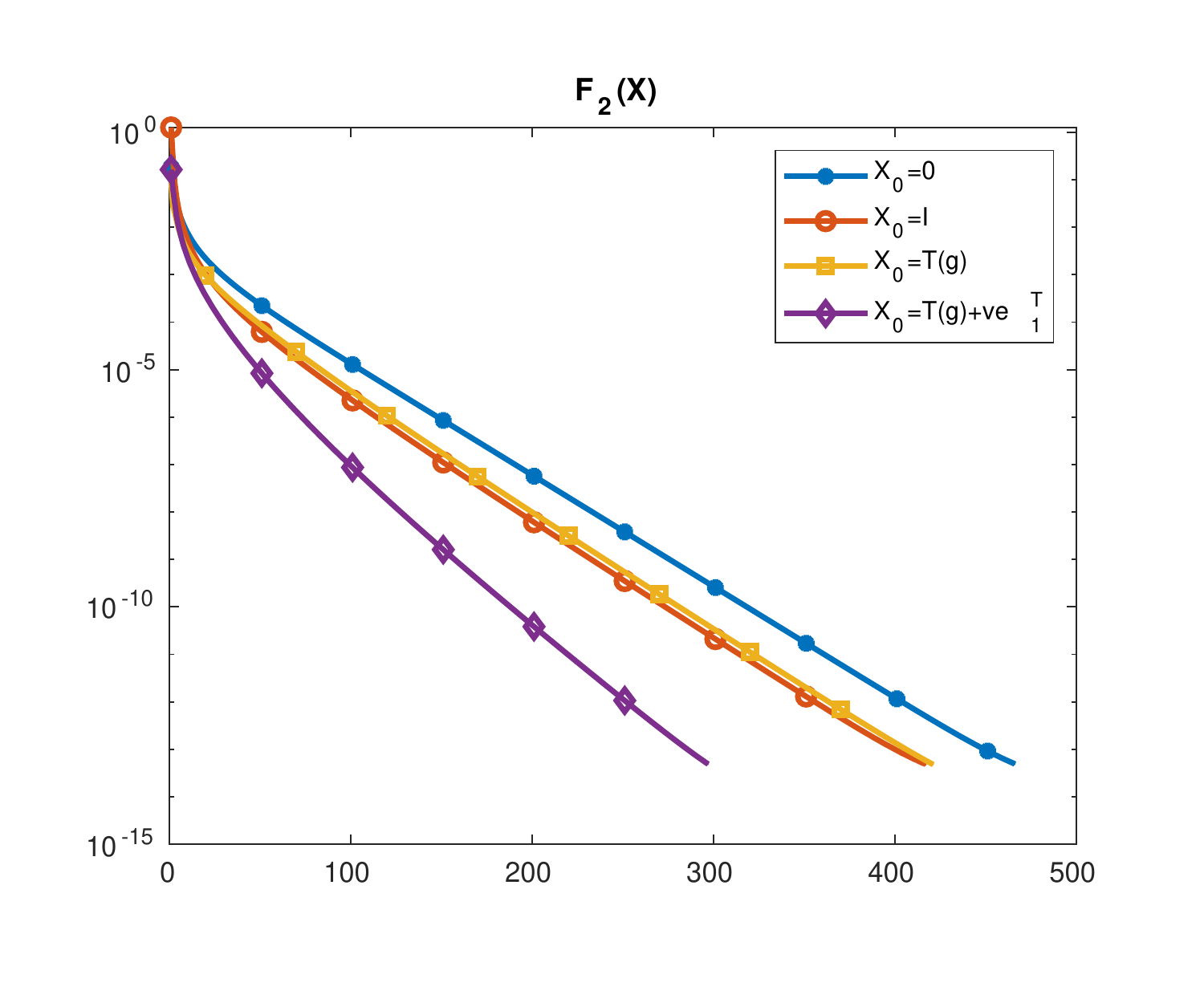}{MT7F2.pdf}
\pgfdeclareimage[width=5cm]{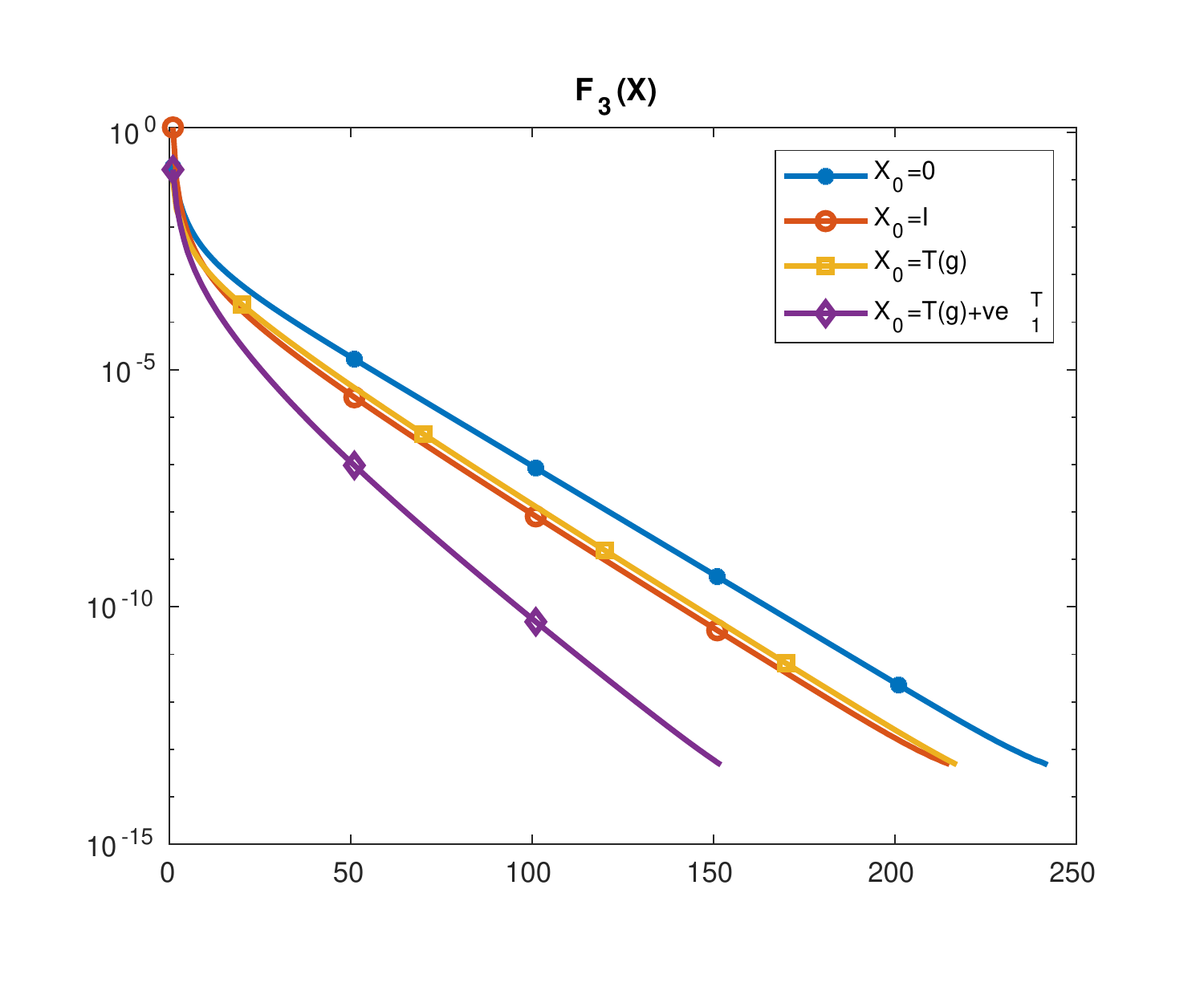}{MT7F3.pdf}
\pgfdeclareimage[width=5cm]{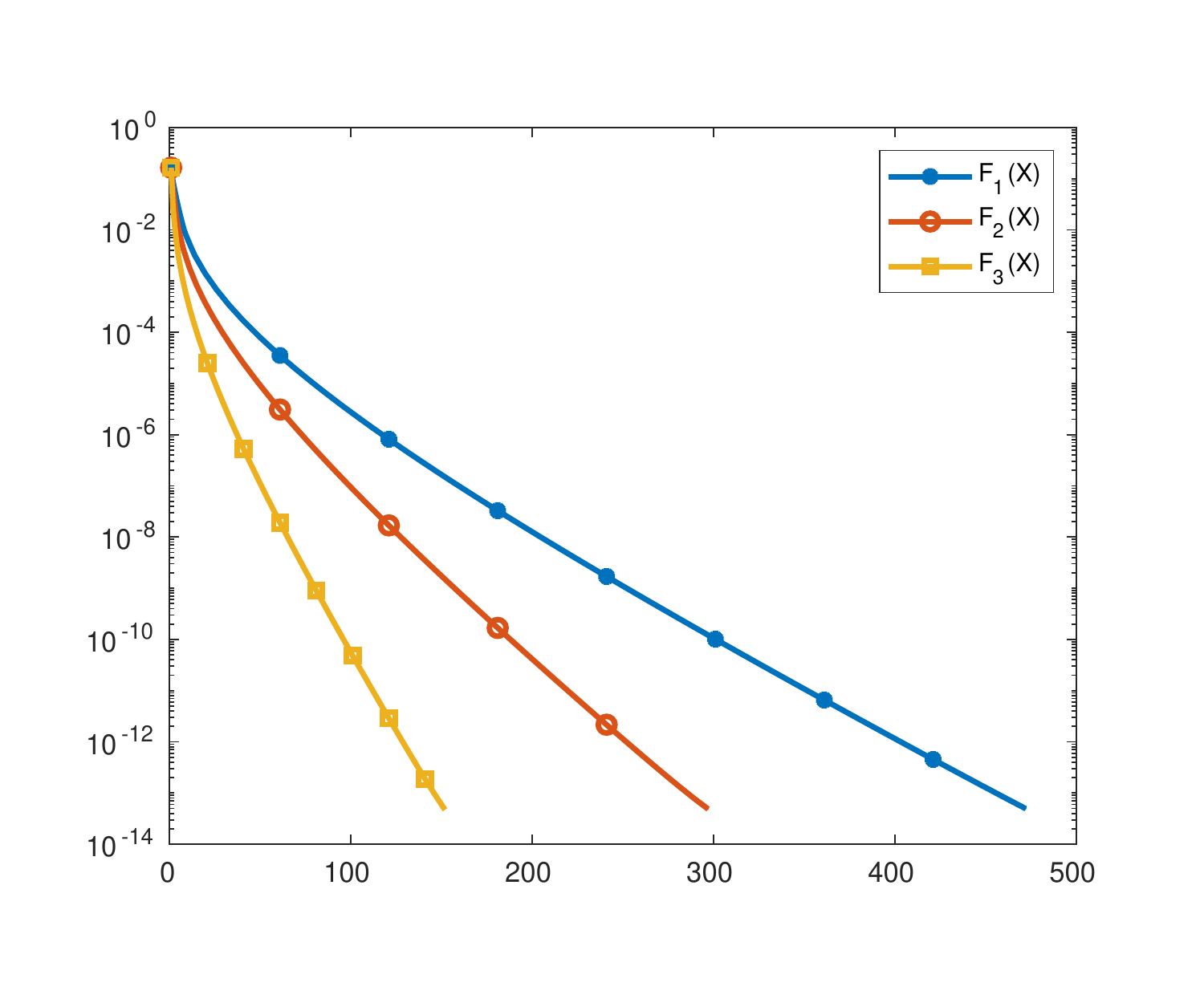}{MT7F123}

\begin{figure}\label{fig:MT7}
  \begin{center}
\begin{tabular}{cc}
\pgfuseimage{MT7F1.pdf}&\pgfuseimage{MT7F2.pdf}\\
\pgfuseimage{MT7F3.pdf}&\pgfuseimage{MT7F123}
\end{tabular}\caption{Two-node Jackson network for Problem $7$ of Table
  \ref{tab:mt}:
  Residual error per step in the three functional iterations $F_1,F_2,F_3$ for different values of the initial matrix $X_0$. In the fourth graph, comparisons of the errors for the three iterations with $X_0=T(g)+v e_1^T$.}
\end{center}
\end{figure}

\begin{table}
  \begin{center}\label{tab:MT7}
    {\footnotesize
    \begin{tabular}{cc}
\begin{tabular}{c|cccc}
&$0$ & $I$ & $T(g)$ & $T(g)+ve_1^T$\\ \hline $F_1$ &  58.6 &  54.7 &
   65.7 &  69.1\\ $F_2$ & 37.2 & 34.8 &  40.0 & 28.9\\ $F_3$ &
   59.5 &  56.2 &  73.3 &  52.2 
\end{tabular}&
\begin{tabular}{c|cccc}
&$0$ & $I$ & $T(g)$ & $T(g)+ve_1^T$\\ \hline
$F_1$ &  735   & 654   & 668   & 472  \\
$F_2$ &  466   & 416   & 421   & 297  \\
$F_3$ &  242    &215    & 217   & 152  
\end{tabular}
    \end{tabular}\caption{Two-node Jackson network for Problem $7$ of Table \ref{tab:mt}: CPU time in
      seconds (left) and number of steps (right) required by the three fixed point
      iteration to arrive at a residual error at most $5.0\cdot 10^{-14}$ starting
      with different values of $X_0$.}
  }\end{center}
  \end{table}

We may observe that in this case CR is the most efficient algorithm
and that the results of Theorems \ref{thm:a1} and \ref{thm:a2} are respected. In
fact the iteration given by $F_3$ with $X_0=T(g)+ve_1^T$ is the
fastest one in terms of number of steps. While concerning the CPU
time, the iteration $F_2$ with $X_0=T(g)+ve_1^T$ is the fastest.

Table \ref{tab:MT11} concerns the case of a Two-node Jackson network
with the choice of parameters given by $\lambda_1=5$, $\lambda_2=
0.7$, $\mu_1= 2$, $\mu_2= 2$, $p= 0.5$ and $q= 0.5$.
In this model,
seen as a random walk in the quarter plane, the overall probability to
move right is higher than the overall probability to move left, while
the probability to move down is higher than the probability to move
up.
%
%
The table reports the CPU time in seconds (left) together
with the number of steps required by the three iterations (right) to
arrive at a residual error at most $5.0\cdot 10^{-14}$. The computation of the
symbol $g(z)$ remains almost inexpensive even though the numerical
length of the coefficient vector of the symbol $g(z)$ is quite large,
in fact the coefficients $g_i$ are computed in less than $0.007$
seconds, and the size of the coefficient vector is $31$ for the coefficients of the negative powers of $z$ and $8424$
for the coefficients of the positive powers, respectively.  The numerical size of the correction is $28\times
6937$.

\begin{table}{\footnotesize
  \begin{center}\label{tab:MT11}
    \begin{tabular}{cc}
\begin{tabular}{c|cccc}
&$0$ & $I$ & $T(g)$ & $T(g)+ve_1^T$\\ \hline 
$F_1$ & 102.8 & 96.0 &  17.5 & 16.0\\
$F_2$ & 49.0 & 46.6 & 9.8 & 9.8\\ 
$F_3$ & 4981.0 & 4502.0 & 997.0 & 913.0   
\end{tabular}&
\begin{tabular}{c|cccc}
&$0$ & $I$ & $T(g)$ & $T(g)+ve_1^T$\\ \hline
$F_1$ &  806 & 738 & 103 & 100\\
$F_2$ &  310 &  285 & 47 &  46\\
$F_3$ &  169  & 149 & 37 &  35  
\end{tabular}
    \end{tabular}\caption{Two-node Jackson network for $\lambda_1=5$,
      $\lambda_2=0.7$, $\mu_1=2 $, $\mu_2=2 $, $p=0.5 $, $q=0.5$:
      CPU time in seconds (left) and number of steps (right) required by
      the three fixed point iteration to arrive at a residual error
      at most $5.0\cdot 10^{-14}$ starting with different values of $X_0$.
      Cyclic reduction requires 8 steps for the overall CPU time of 70.8
      seconds.}
  \end{center}}
  \end{table}

For this problem, cyclic reduction provides the solution in 8 steps
and in $70.8$ seconds, while Newton iteration requires 8 steps and
takes 782 seconds.  In this case, due to the large size of the
matrices involved in the computation of matrix inverses, CR takes a
much larger time than the simple functional iterations given by $F_1$
and $F_2$ which either do not involve inversion or require just only
one matrix inversion. While iteration given by $F_2$ takes less than
10 seconds, the iteration given by $F_3$, even though  is the
fastest in terms of number of steps, needs a large CPU time. In fact,
similarly to CR, it requires a matrix inversion at each step which
becomes more expensive as the approximation approaches the
limit. Newton iteration has the same convergence features as CR,
however, the larger cost of solving a Sylvester equation makes this
iteration much slower than the other ones at least for this problem.

In this model, increasing the values of $\lambda_1$ makes  the
size of the output much larger.  In particular, with values $\lambda_1\ge 6$,
cyclic reduction breaks down for memory overflow while $F_2(X)$
provides the solution with a slight increase of the CPU time.

Figure \ref{fig:MT11G} provides the solution $G$ in log scale where
the Toeplitz part and the correction parts are separately represented.

\pgfdeclareimage[width=5cm]{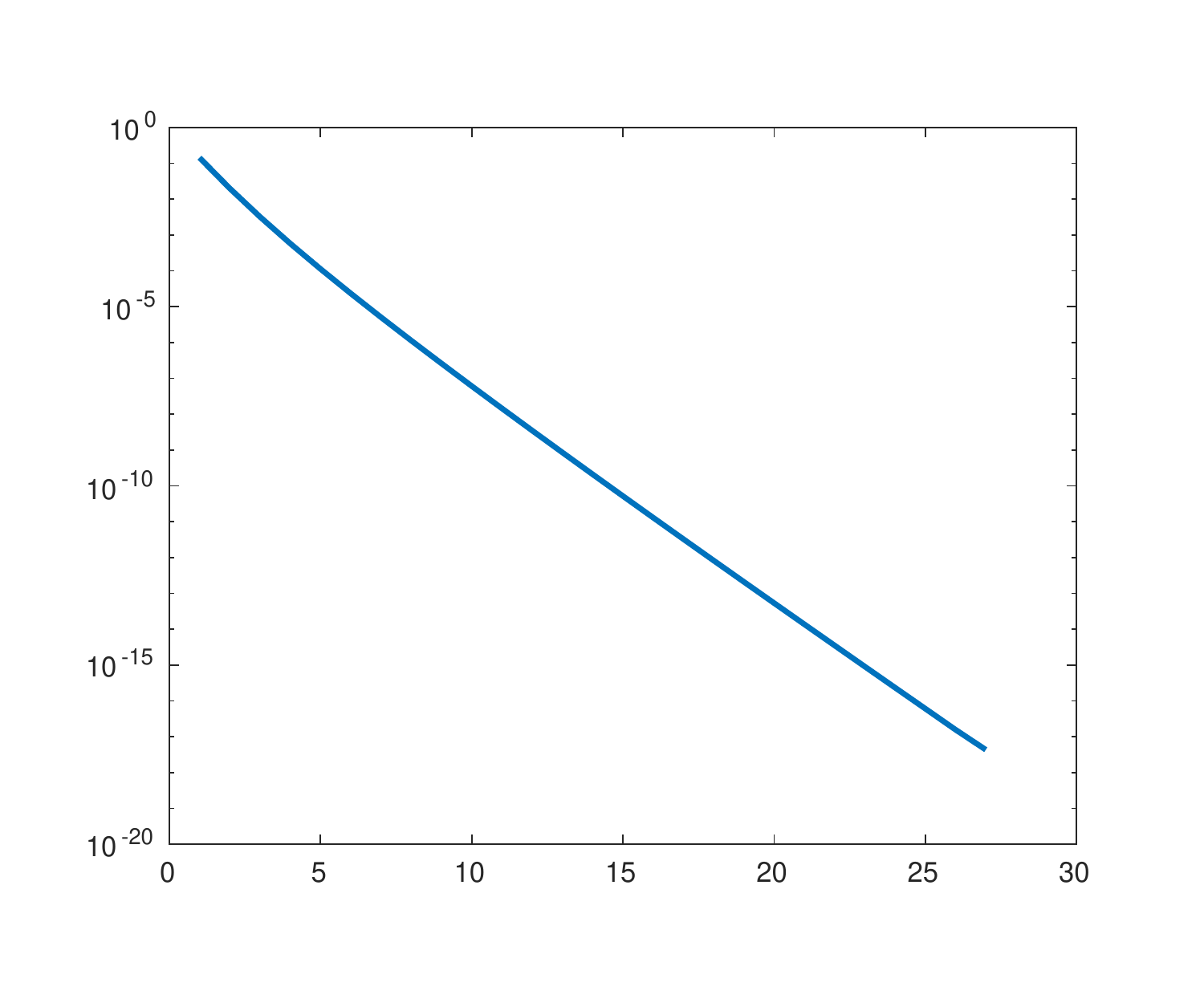}{MT11gm.pdf}
\pgfdeclareimage[width=5cm]{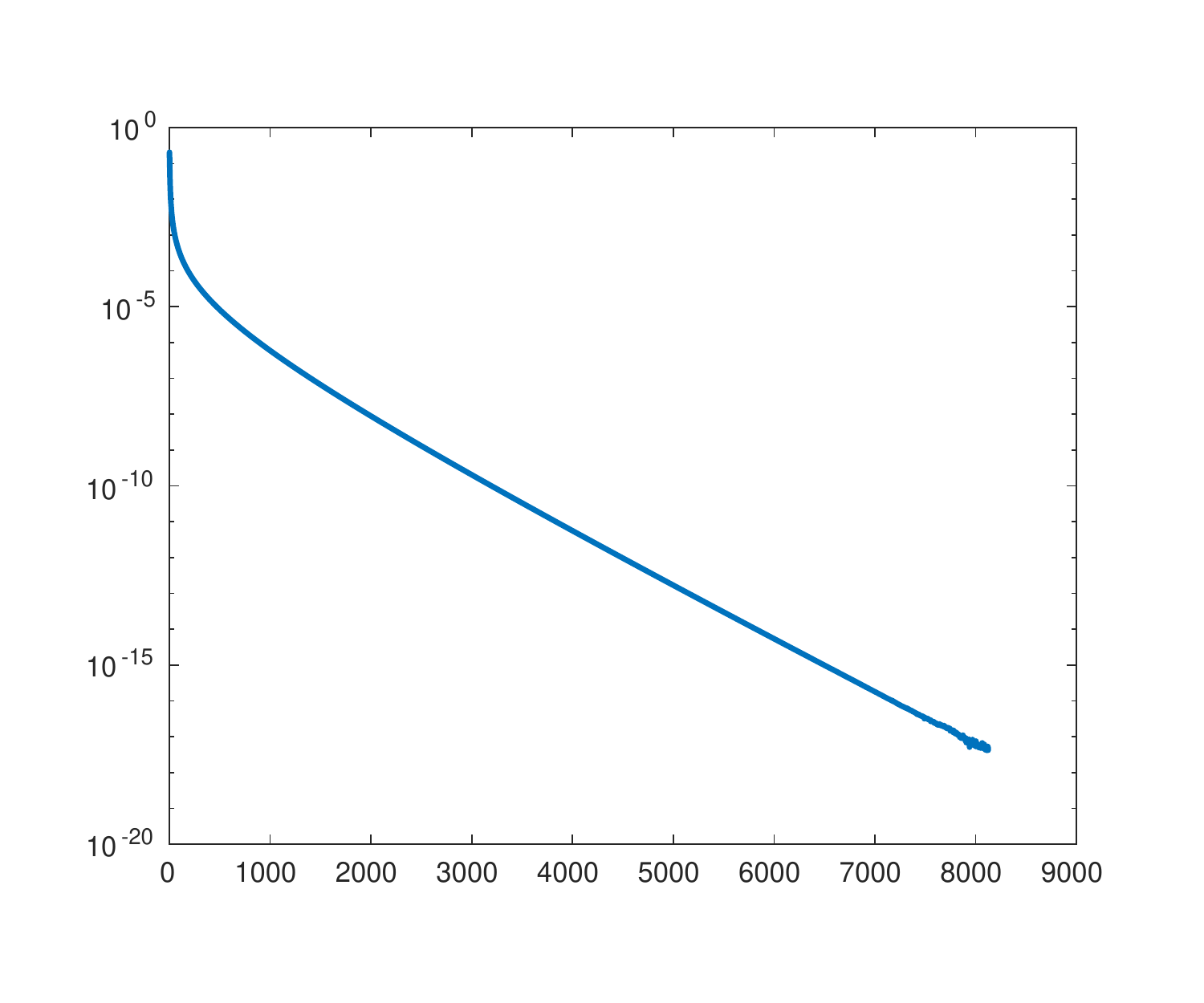}{MT11gp.pdf}
\pgfdeclareimage[width=6cm]{MT11Gc1.png}{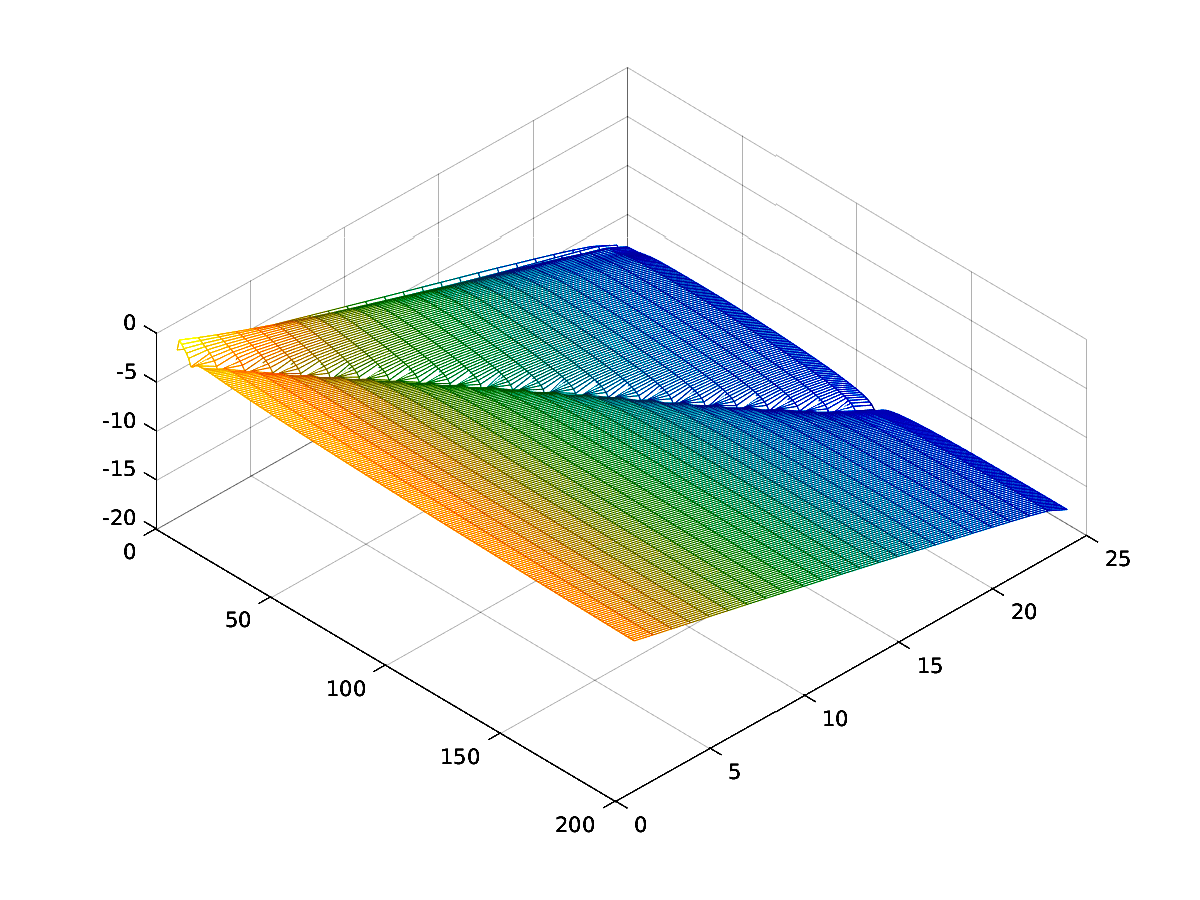}
\pgfdeclareimage[width=6cm]{MT11Gc2.png}{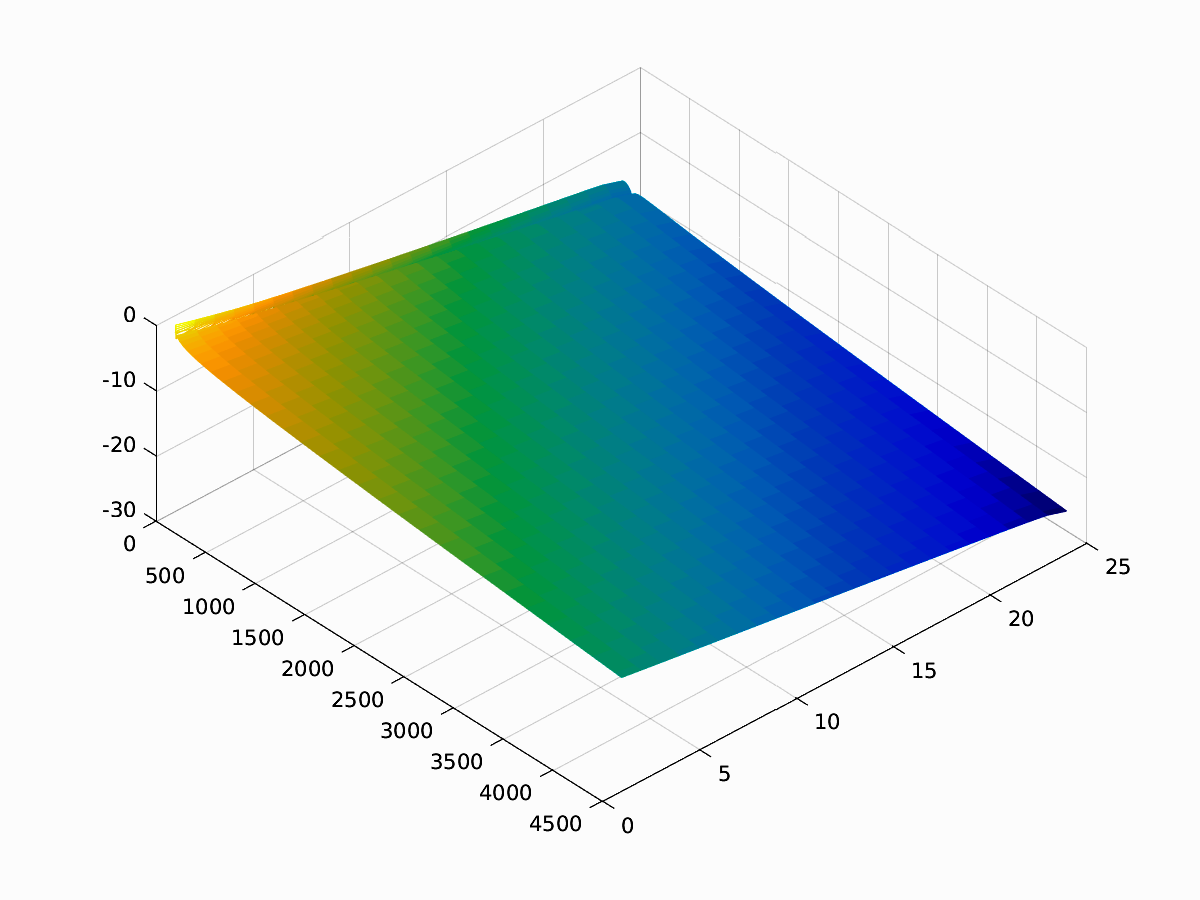}

\begin{figure}\label{fig:MT11G}
  \begin{center}
\begin{tabular}{cc}
  \pgfuseimage{MT11gm.pdf}&\pgfuseimage{MT11gp.pdf}\\
    \pgfuseimage{MT11Gc2.png}&\pgfuseimage{MT11Gc1.png}\\
\end{tabular}
  \end{center}\caption{Two-node Jackson network for $\lambda_1=5$,
    $\lambda_2=0.7$, $\mu_1=2 $, $\mu_2=2 $, $p=0.5 $, $q=0.5$:
    Solution $G$. In the upper part, the log-scale graph of the
    coefficients $g_i$ of the symbol for $i\le 0$ (left) and for
    $i\ge 0$ (right).
    In the lower part the log-scale graph of the absolute value of the whole correction
    (left) together with a zoom (right).}
  \end{figure}

\subsection{Assistance from idle server}
Here we consider a class of queueing models for a system with two
servers and two queues.  Arrivals to queues 1 and 2 occur as
independent Poisson processes with parameters $\lambda_1$ and
$\lambda_2$, respectively. The service times of servers 1 and 2 are
exponentially distributed with parameters $\mu_1$ and $\mu_2$
respectively.  Each server serves its own queue according to a
first-come-first-served discipline. If one of the queues is empty, the
server for that queue assists the other server, doubling the latter's
service rate. If there is an arrival to a queue while its server is
assisting the other queue, the server immediately ceases assisting and
serves its own queue.  This stochastic process is ergodic if and only if
$\rho_1 + \rho_2 < 2$, where $\rho_i=\lambda_i/\mu_i$, $i=1,2$.

For this model, the matrices $A_{-1},A_0,A_1$ are given by
$A_{-1}=\hbox{diag}(2\mu_1,\mu_1,\mu_1, \ldots)$,
$A_0=\hbox{trid}(\mu_2,-\lambda_1-\lambda_2-\mu_1-\mu_2,\lambda_2)+
(\mu_2-\mu_1)e_1e_1^T$, $A_1=\lambda_1 I$.

In this example, we have chosen the values of the parameters in order that
the numerical size of the matrix $G$ is substantially large, namely,
$\lambda_1=0.01 $, $\lambda_2=2.9$, $\mu_1=0.03$, $\mu_2=2.0$.

The situation is analogous to that of the second example in Section
\ref{sec:2nJn} where the methods based on functional iterations
perform better than cyclic reduction and Newton iteration. The graphs
in Figure \ref{fig:idle} and the values in Table \ref{tab:idle}
synthesize the behavior of the algorithms. Cyclic reduction takes about 17 seconds of CPU while Newton iteration about 600 seconds. Slightly increasing the value of $\lambda_1$, CR breaks down for memory overflow while functional iterations still compute correctly the solution $G$.

\pgfdeclareimage[width=5cm]{IDLF1}{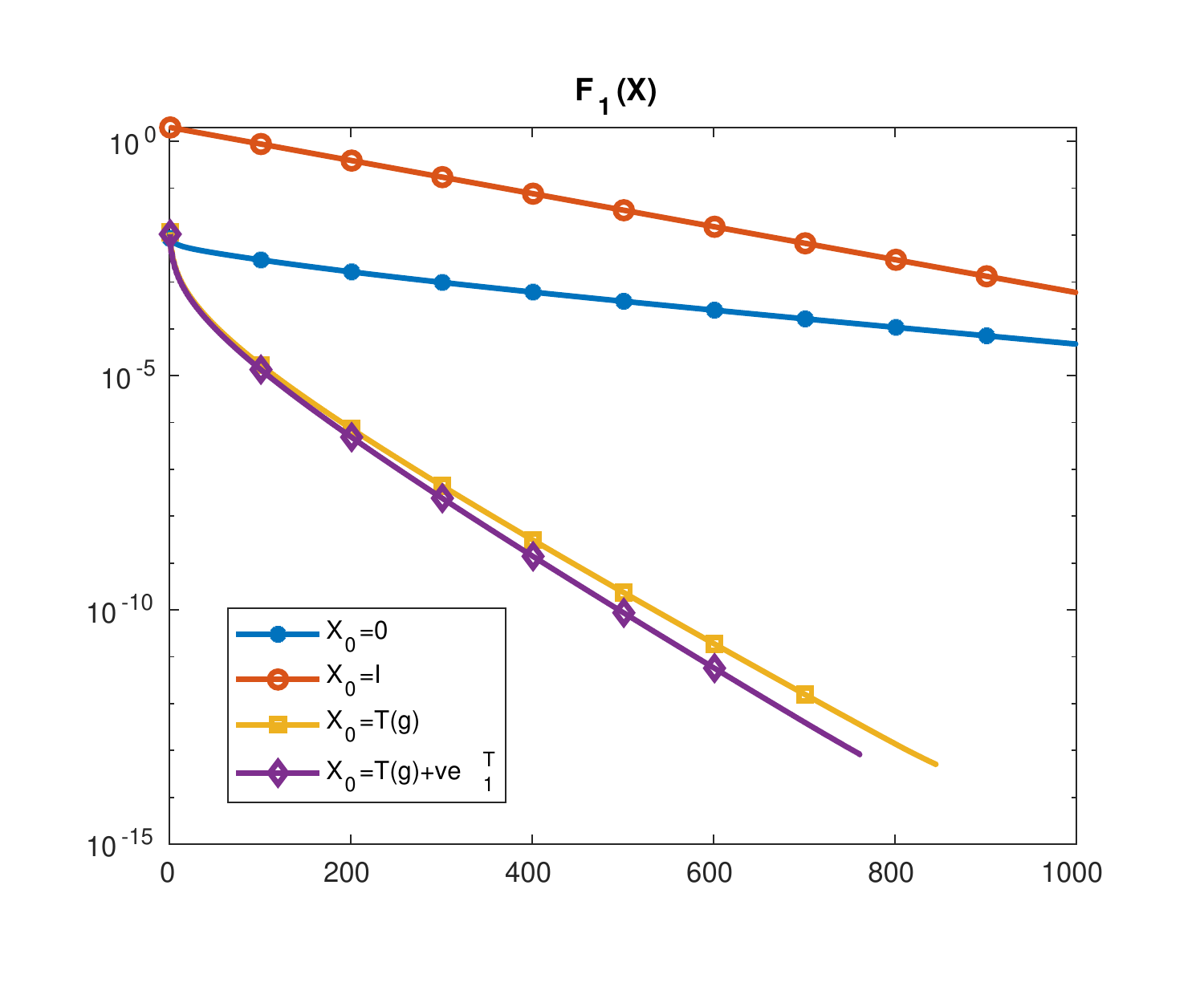}
\pgfdeclareimage[width=5cm]{IDLF2}{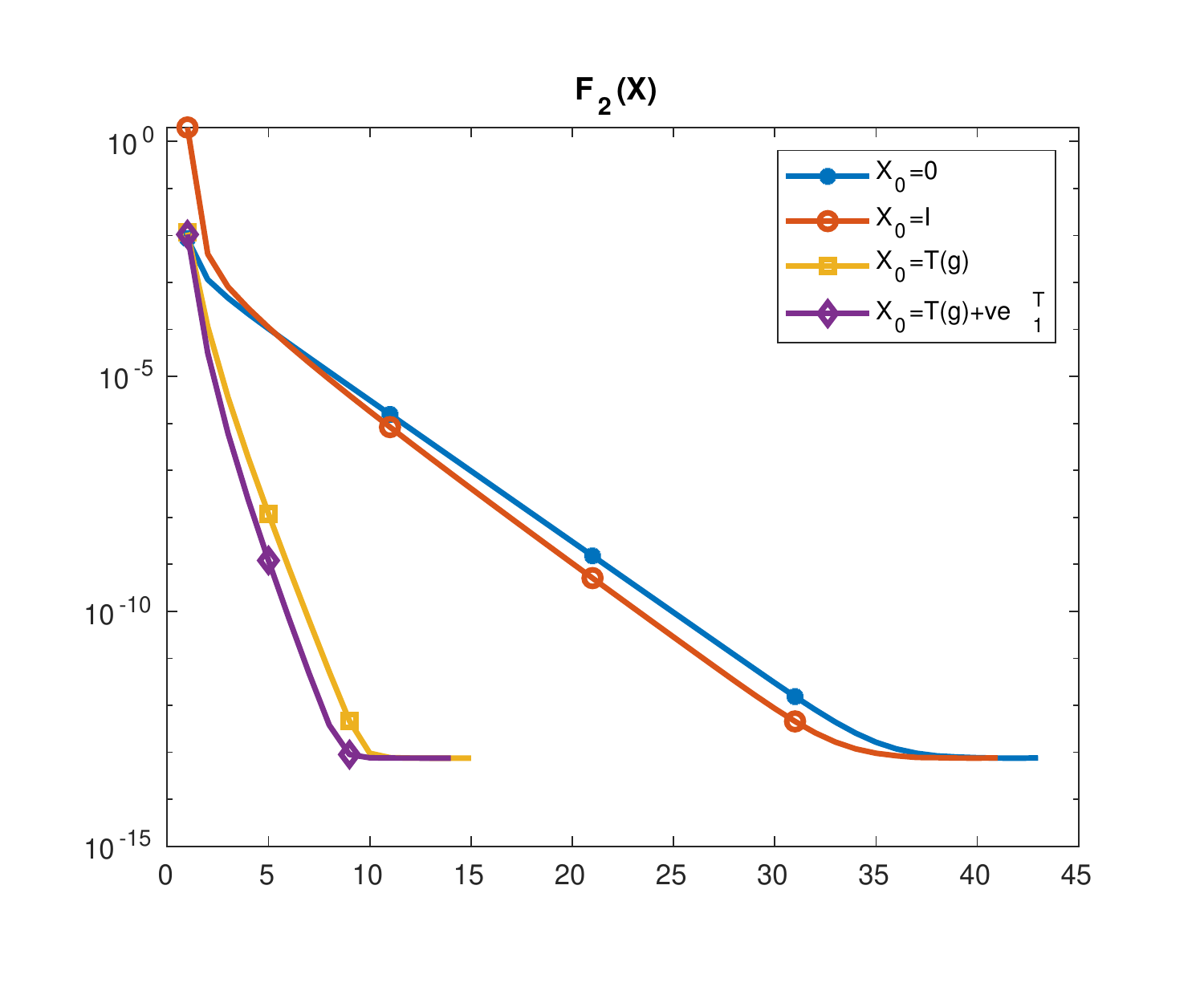}
\pgfdeclareimage[width=5cm]{IDLF3}{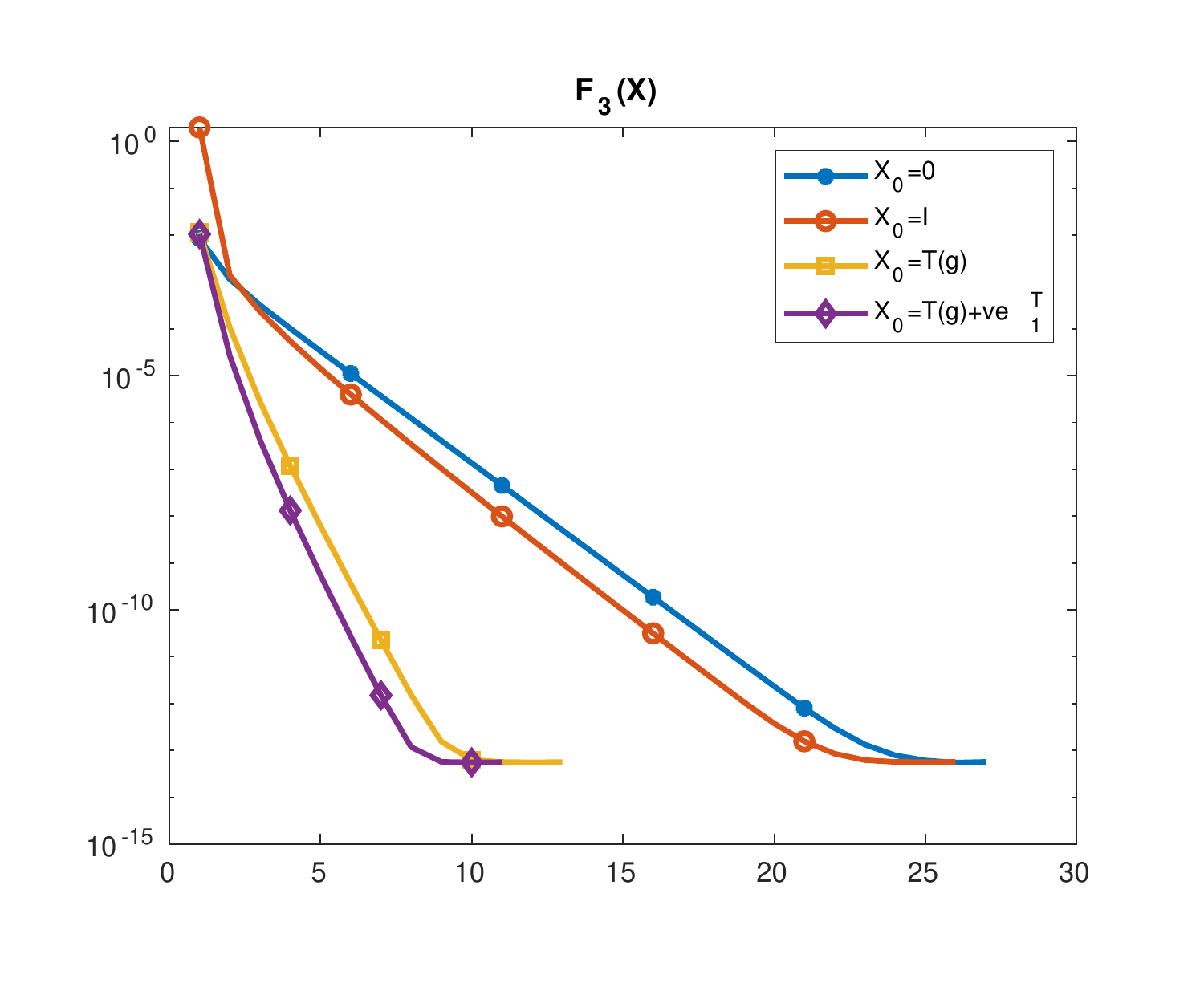}
\pgfdeclareimage[width=5cm]{IDLF123}{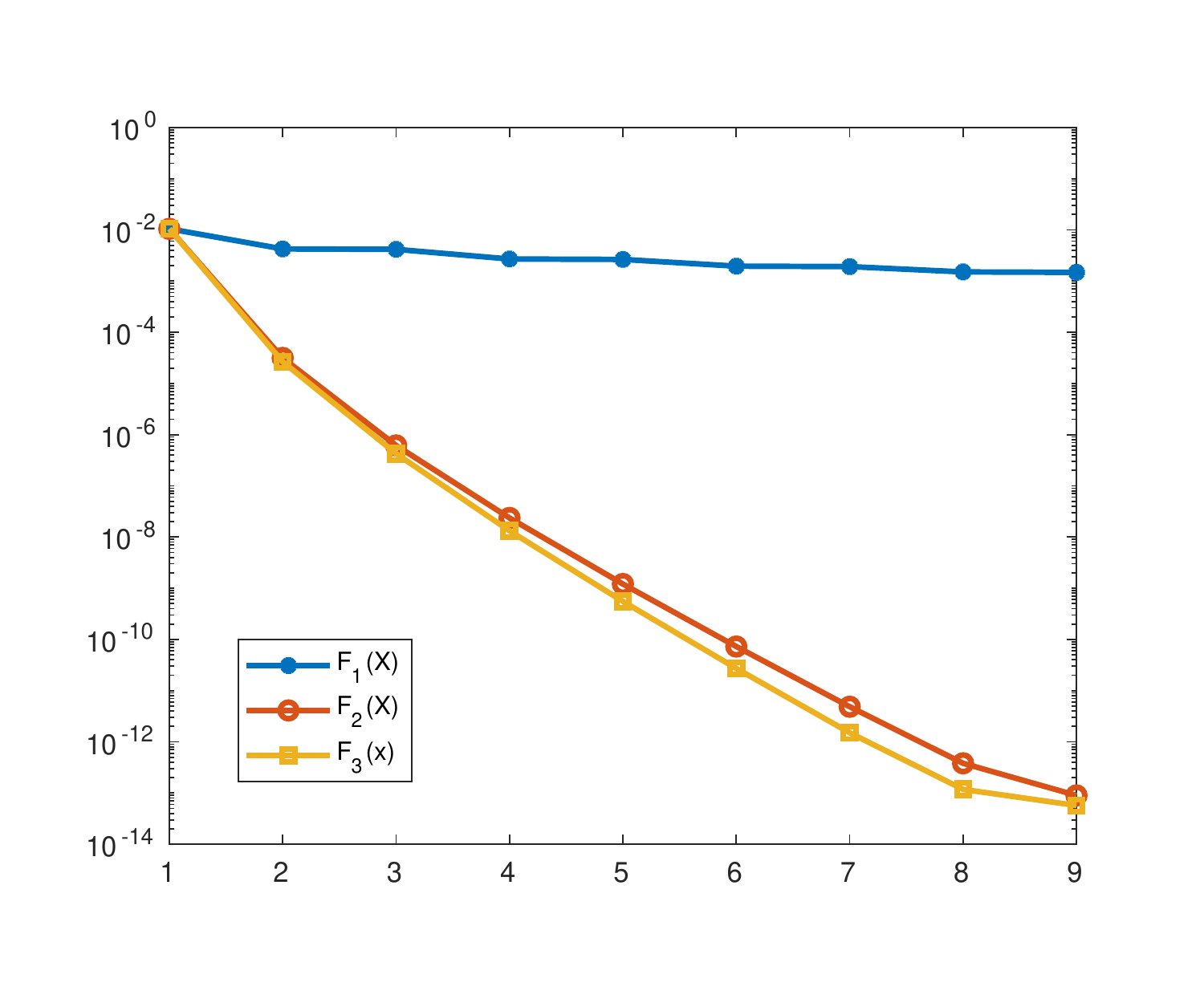}

\begin{figure}\label{fig:idle}
  \begin{center}
\begin{tabular}{cc}
\pgfuseimage{IDLF1}&\pgfuseimage{IDLF2}\\
\pgfuseimage{IDLF3}&\pgfuseimage{IDLF123}
\end{tabular}\caption{Assistance from idle server model from \cite{Motyer-Taylor} for $\lambda_1=0.01$,
  $\lambda_2=2.9$, $\mu_1=0.03 $, $\mu_2=2.0 $:
  Residual error per step in the three functional iterations
  $F1,F_2,F_3$ for different values of the initial matrix $X_0$.
  In the fourth graph, comparisons of the errors for the three
  iterations with $X_0=T(g)+ve_1^T$.}
\end{center}
\end{figure}

\begin{table}
  \begin{center}\label{tab:idle}
  {\footnotesize
    \begin{tabular}{cc}
\begin{tabular}{c|cccc} 
&$0$ & $I$ & $T(g)$ & $T(g)+ve_1^T$\\ \hline 
$F_1$ & * & * &204.8   & 191.6\\ 
$F_2$ & 16.4 & 13.7 & 4.6 & 3.9\\ 
$F_3$ & 405.1 & 402.6 & 218.3 & 183.6
\end{tabular}&
\begin{tabular}{c|cccc}  
&$0$ & $I$ & $T(g)$ & $T(g)+ve_1^T$\\ \hline
$F_1$ &  * & * & 844 & 782\\
$F_2$ &  42 &  40 & 10 & 9 \\
$F_3$ &  26 &  25 & 9 & 7
\end{tabular}
    \end{tabular}\caption{Assistance from idle server model from \cite{Motyer-Taylor} for $\lambda_1=0.01$,
  $\lambda_2=2.9$, $\mu_1=0.03 $, $\mu_2=2.0 $:
      CPU time in seconds (left) and number of steps (right) required by
      the three fixed point iterations to arrive at a residual error
      at most $5.0\cdot 10^{-14}$ starting with different values of $X_0$. A ``$*$'' denotes a number of steps greater than 1000. Cyclic reduction requires 5 steps and 17 seconds of CPU time.}
  }\end{center}
\end{table}



\subsection{Random walk in the quarter plane}
Here we consider an example where the condition $A_{-1}\one>A_1\one$
is  satisfied everywhere exept in the first component. The example describes a
random walk in the quarter plane where a particle can occupy positions
in a grid and we know the probabilities that the particle moves to the
neighboring positions.  In order to better describe the test problem, we
denote $H=(h_{i,j})_{i,j=-1,1}$ the matrix with the probabilities of
transition in the inner part of the quarter plane, while we denote
$Y=(y_{i,j})$ the $3\times 2$ matrix with the probabilities of
transition in the $y$ axis. These two matrices fully describe the
coefficients $A_i$ which can be written as $A_i=T(a_i)+E_i$ where
$a_i(z)=\sum_{j=-1}^1 h_{i,j}z^j$ and $E_i=e_1[y_{i,0}-h_{i,0},
  y_{i,1}-h_{i,1}, 0,\ldots]$, compare with \eqref{eq:ab}.

 The random walk of this example is obtained with the values
\[
H=\frac19\begin{bmatrix}1 & 0& 1\\ 2& 0& 0\\ 2& 2&1
\end{bmatrix},\quad Y=\frac1{3}\begin{bmatrix}
1&1\\ 0&1\\0&0
\end{bmatrix}
\]
so that the condition $A_{-1}\one>A_1\one$ is satisfied in all the components but the first.

For this problem, there exist two nonnegative solutions $G$ and
$\widehat G$ to equation \eqref{eq:G} which satisfy the inequality
$G\le\widehat G$. Moreover $\widehat G$ is stochastic while $G$ is
substochastic. The two solutions have the same symbol $g(z)$ and differ only for the correction part.
Starting with $X_0=0$ or with $X_0=T(g)$ the sequences
generated by the functional iterations converge to $G$. Starting with $X_0=I$ or $X_0=T(g)+ve_1^T$ the sequences converge to $\widehat
G$, while CR and Newton iteration converge to $G$.

The results of this test are summarized in Table \ref{tab:last}. Both CR and Newton iteration take 23 steps in order to arrive at
numerical convergence. However CR takes more than 8  minutes of CPU time
while Newton iteration just 3.4 seconds. The large amount of CPU time
taken by CR is due to the fact that the inverse matrices involved at
each step of CR are QT matrices with a correction having a size which
increases step after step and reaches values larger than $10^6$, while
Newton iteration involves QT matrices with corrections having almost
the same size of the correction of $G$ which is $126\times 36$.
The growth of the sizes of the correction matrices in the algorithms
is an issue which deserves further analysis.

Functional iterations $F_1,F_2$ and $F_3$ with $X_0=0$ or with
$X_0=T(g)$ take a large number of steps to converge numerically to $G$
while with $X_0=I$ or $X_0=T(g)+ve_1^T$ the number of iterations is
much smaller but the limit of the sequences is the stochastic solution 
$\widehat G$ which is not the minimal one.

For this problem, the combination of few steps of CR followed by few steps of Newton iteration provide a substantial acceleration in terms of CPU time.
In fact, the first iterations of CR, involving matrices of small size, have a low cost. The last few steps of CR, which have a much higher cost, are replaced by Newton steps.

\begin{table}\label{tab:last}
  \begin{center}
  \begin{tabular}{c|ccc|ccc|ccc}
      &$F_1$&$F_2$&$F_3$&$F_1$&$F_2$&$F_3$&CR&Newton&CR+Newton\\ \hline
      iter& *& *&*& 285&205&119&23& 23&15+10\\
      CPU & *& *&*& 3.1&2.3&2.6& 524&3.4 &0.4+1.5\\
  \end{tabular}\caption{Random walk in the quarter plane: Number of iterations and CPU time in seconds. From left to right: fixed point iterations with $X_0=T(g)$, fixed point iterations with $X_0=T(g)+ve_1^T$, cyclic reduction, Newton iteration with $X_0=0$, combination of cyclic reduction and Newton iteration. A 
``$*$'' denotes more than 10000 iterations and a CPU larger than 1000 seconds. Starting the iterations with $X_0=T(g)+ve_1^T$ generates sequences converging to the stochastic solution $\widehat G$, while starting with $X_0=0$ or applying CR, Newton iteration and their combinations generate sequences converging to the minimal (substochastic) solution $G$.}
\end{center}
\end{table}

\section{Conclusions}
We have analyzed quadratic matrix equations encountered in the solution of random walk in the quarter plane where the solution of interest is the minimal nonnegative solution $G$. This class of equations is characterized by matrix coefficients with infinite size which belong to the class $\mathcal{QT}$ of Quasi-Toeplitz matrices. We have  provided a perturbation analysis of $G$, introduced some fixed point algorithms for computing $G$ and compared their convergence speed. The algorithms rely on the properties of $\mathcal{QT}$  matrices recently investigated in \cite{cqttoolbox}. 
Numerical experiments show that in many cases the CPU time and the memory resources required by our approach are significantly inferior to the ones required by the algorithm of cyclic reduction, which is considered as the algorithm of choice for this class of problems. The effectiveness of Newton iteration depends on the growth of the sizes of the correction part in the QT matrices generated by the algorithm.


\begin{thebibliography}{10}

\bibitem{bs}
R.~Bartels and G.~Stewart.
\newblock Solution of the matrix equation {$AX+XB=C$}: {A}lgorithm 432.
\newblock {\em Comm. ACM}, 15:820--826, 1972.

\bibitem{blm:book}
D.~A. Bini, G.~Latouche, and B.~Meini.
\newblock {\em Numerical methods for structured {M}arkov chains}.
\newblock Numerical Mathematics and Scientific Computation. Oxford University
  Press, New York, 2005.
\newblock Oxford Science Publications.

\bibitem{bmm}
D.~A. Bini, S.~Massei, and B.~Meini.
\newblock Semi-infinite quasi-{T}oeplitz matrices with applications to {QBD}
  stochastic processes.
\newblock {\em Math. Comp.}, 87(314):2811--2830, 2018.

\bibitem{bmmr}
D.~A. Bini, S.~Massei, B.~Meini, and L.~Robol.
\newblock On quadratic matrix equations with infinite size coefficients
  encountered in {QBD} stochastic processes.
\newblock {\em Numer. Linear Algebra Appl.}, 25(6):2128, 12, 2018.

\bibitem{bmmr19}
D.~A. Bini, S.~Massei, B.~Meini, and L.~Robol.
\newblock Matrix analytic methods for reflected random walks with restart.
\newblock {\em In preparation}, 2019.

\bibitem{cqttoolbox}
D.~A. Bini, S.~Massei, and L.~Robol.
\newblock Quasi-{T}oeplitz matrix arithmetic: a {MATLAB} toolbox.
\newblock {\em Numerical Algorithms}, 81(2):741--769, 2019.

\bibitem{bm:cr}
D.~A. Bini and B.~Meini.
\newblock On the solution of a nonlinear matrix equation arising in queueing
  problems.
\newblock {\em SIAM J. Matrix Anal. Appl.}, 17(4):906--926, 1996.

\bibitem{pwcr}
D.~A. Bini and B.~Meini.
\newblock Improved cyclic reduction for solving queueing problems.
\newblock {\em Numer. Algorithms}, 15(1):57--74, 1997.

\bibitem{fayolle:book}
G.~Fayolle, R.~Iasnogorodski, and V.~Malyshev.
\newblock {\em Random walks in the quarter plane}, volume~40 of {\em
  Probability Theory and Stochastic Modelling}.
\newblock Springer, Cham, second edition, 2017.
\newblock Algebraic methods, boundary value problems, applications to queueing
  systems and analytic combinatorics.

\bibitem{flatto}
L.~Flatto and S.~Hahn.
\newblock Two parallel queues created by arrivals with two demands {I}.
\newblock {\em SIAM Journal on Applied Mathematics}, 44(5):1041--1053, 1984.

\bibitem{haque}
L.~Haque, Y.~Q. Zhao, and L.~Liu.
\newblock Sufficient conditions for a geometric tail in a {QBD} process with
  many countable levels and phases.
\newblock {\em Stochastic Models}, 21(1):77--99, 2005.

\bibitem{henrici}
P.~Henrici.
\newblock {\em Applied and computational complex analysis. {V}ol. 1}.
\newblock Wiley Classics Library. John Wiley \& Sons, Inc., New York, 1988.
\newblock Power series---integration---conformal mapping---location of zeros,
  Reprint of the 1974 original, A Wiley-Interscience Publication.

\bibitem{HK}
N.~J. Higham and H.-M. Kim.
\newblock Numerical analysis of a quadratic matrix equation.
\newblock {\em IMA J. Numer. Anal.}, 20(4):499--519, 2000.

\bibitem{koba-miya}
M.~Kobayashi and M.~Miyazawa.
\newblock Revisiting the tail asymptotics of the double {QBD} process:
  refinement and complete solutions for the coordinate and diagonal directions.
\newblock In {\em Matrix-analytic methods in stochastic models}, volume~27 of
  {\em Springer Proc. Math. Stat.}, pages 145--185. Springer, New York, 2013.

\bibitem{kroese}
D.~P. Kroese, W.~R.~W. Scheinhardt, and P.~G. Taylor.
\newblock Spectral properties of the tandem {J}ackson network, seen as a
  quasi-birth-and-death process.
\newblock {\em Ann. Appl. Probab.}, 14(4):2057--2089, 2004.

\bibitem{latoucheN}
G.~Latouche.
\newblock Newton's iteration for non-linear equations in {M}arkov chains.
\newblock {\em IMA J. Numer. Anal.}, 14(4):583--598, 1994.

\bibitem{lat:varese}
G.~Latouche, S.~Mahmoodi, and P.~Taylor.
\newblock Level-phase independent stationary distributions for {GI/M/1}-type
  {M}arkov chains with infinitely-many phases.
\newblock {\em Performance Evaluation}, 70(9):551--563, 2013.

\bibitem{latouche11}
G.~Latouche, G.~T. Nguyen, and P.~G. Taylor.
\newblock Queues with boundary assistance: the effects of truncation.
\newblock {\em Queueing Syst.}, 69(2):175--197, 2011.

\bibitem{lr:book}
G.~Latouche and V.~Ramaswami.
\newblock {\em Introduction to matrix analytic methods in stochastic modeling}.
\newblock ASA-SIAM Series on Statistics and Applied Probability. Society for
  Industrial and Applied Mathematics (SIAM), Philadelphia, PA; American
  Statistical Association, Alexandria, VA, 1999.

\bibitem{latouche02}
G.~Latouche and P.~Taylor.
\newblock Truncation and augmentation of level-independent {QBD} processes.
\newblock {\em Stochastic Process. Appl.}, 99(1):53--80, 2002.

\bibitem{meini}
B.~Meini.
\newblock New convergence results on functional iteration techniques for the
  numerical solution of {$M/G/1$} type {M}arkov chains.
\newblock {\em Numer. Math.}, 78(1):39--58, 1997.

\bibitem{miya1}
M.~Miyazawa.
\newblock Light tail asymptotics in multidimensional reflecting processes for
  queueing networks.
\newblock {\em TOP}, 19(2):233--299, 2011.

\bibitem{miya-zhao}
M.~Miyazawa and Y.~Q. Zhao.
\newblock The stationary tail asymptotics in the {$GI/G/1$}-type queue with
  countably many background states.
\newblock {\em Adv. in Appl. Probab.}, 36(4):1231--1251, 2004.

\bibitem{Motyer-Taylor}
A.~J. Motyer and P.~G. Taylor.
\newblock Decay rates for quasi-birth-and-death processes with countably many
  phases and tridiagonal block generators.
\newblock {\em Adv. Appl. Prob.}, 38:522--544, 2006.

\bibitem{neuts}
M.~F. Neuts.
\newblock {\em Matrix-geometric solutions in stochastic models: An algorithmic
  approach}, volume~2 of {\em Johns Hopkins Series in the Mathematical
  Sciences}.
\newblock Johns Hopkins University Press, Baltimore, Md., 1981.

\bibitem{ozawa19}
T.~Ozawa.
\newblock Stability condition of a two-dimensional qbd process and its
  application to estimation of efficiency for two-queue models.
\newblock {\em Performance Evaluation}, 130:101 -- 118, 2019.

\bibitem{ozawa18}
T.~Ozawa and M.~Kobayashi.
\newblock Exact asymptotic formulae of the stationary distribution of a
  discrete-time two-dimensional {QBD} process.
\newblock {\em Queueing Systems}, 90(3):351--403, Dec 2018.

\bibitem{robol}
L.~Robol.
\newblock Rational {K}rylov and {ADI} iteration for infinite size
  quasi-{T}oeplitz matrix equations.
\newblock {\em arXiv:1907.02753}, pages 1--24, 2019.

\bibitem{sakuma}
Y.~Sakuma and M.~Miyazawa.
\newblock On the effect of finite buffer truncation in a two-node {J}ackson
  network.
\newblock {\em J. Appl. Probab.}, 42(1):199--222, 2005.

\bibitem{stanford}
D.~Stanford, W.~Horn, and G.~Latouche.
\newblock Tri-layered {QBD} processes with boundary assistance for service
  resources.
\newblock {\em Stochastic Models}, 22(3):361--382, 2006.

\bibitem{taka}
Y.~Takahashi, K.~Fujimoto, and N.~Makimoto.
\newblock Geometric decay of the steady-state probabilities in a
  quasi-birth-and-death process with a countable number of phases.
\newblock {\em Communications in Statistics. Part C: Stochastic Models},
  17(1):1--24, 2001.

\end{thebibliography}

\end{document}